\documentclass[a4paper,11pt,reqno]{article}
\usepackage[margin=1.2in]{geometry}

\usepackage{color, colortbl}

\usepackage{amsmath}
\usepackage{amsfonts}
\usepackage{amsthm}
\usepackage{amssymb}
\usepackage{tikz}
\usepackage[all]{xy}
\usepackage{tikz-cd}
\usepackage{enumerate}
\usepackage{mathtools}
\usepackage{setspace}
\usepackage{tocloft}
\usepackage{hyperref} 

\usepackage{comment}

\usepackage{float}

\definecolor{lightgray}{gray}{0.9}
\definecolor{darkgreen}{rgb}{0,0.5,0}
\definecolor{darkblue}{rgb}{0,0.1,0.5}

\hypersetup{
    colorlinks=true, 
    linkcolor=darkblue, 
    urlcolor=blue, 
    linktoc=all, 
    citecolor=darkgreen
    }

\usepackage{thmtools}
\newtheoremstyle{introTheorems}
  {\topsep}
  {\topsep}
  {\itshape}
  {0pt}
  {\bfseries}
  {}
  { }
  {\thmname{#1}
  \textnormal{\thmnote{#3}.}
  }

\usepackage[abbrev,msc-links, alphabetic, backrefs]{amsrefs}
\usepackage{amsrefs}

\BibSpec{arXiv}{%
  +{}{\PrintAuthors}{author}
  +{,}{ \textit}{title}
  +{,}{ arXiv }{eprint}
}
\BibSpec{other}{%
  +{}{\PrintAuthors}{author}
  +{,}{ \textit}{title}
}

\newtheorem{theorem}{Theorem}[section]

\newtheorem{conjecture}[theorem]{Conjecture}

\newtheorem{definition}[theorem]{Definition}
\newtheorem{example}[theorem]{Example}
\newtheorem{lemma}[theorem]{Lemma}
\newtheorem{problem}[theorem]{Problem}
\newtheorem{question}[theorem]{Question}

\theoremstyle{introTheorems}

\usepackage[abbrev,msc-links, alphabetic, backrefs]{amsrefs}
\usepackage{amsrefs}
\usepackage{doi}

\renewcommand{\PrintDOI}[1]{\doi{#1}}

\BibSpec{arXiv}{%
  +{}{\PrintAuthors}{author}
  +{,}{ \textit}{title}
  +{}{ \parenthesize}{date}
  +{,}{ arXiv }{eprint}
}

\newcommand{\id}{{\rm id}}

\newcommand{\Rep}{{\rm Rep}}
\newcommand{\CoRep}{{\rm CoRep}}

\renewcommand{\dim}{{\rm dim}}
\newcommand{\Hom}{{\rm Hom}}
\newcommand{\Ext}{{\rm Ext}}

\newcommand{\Vect}{{\rm Vect}}

\newcommand{\cB}{\mathcal{B}}
\newcommand{\cC}{\mathcal{C}}
\newcommand{\cD}{\mathcal{D}}

\newcommand{\cH}{\mathcal{H}}
\newcommand{\cM}{\mathcal{M}}
\newcommand{\cN}{\mathcal{N}}
\newcommand{\cL}{\mathcal{L}}
\newcommand{\cZ}{\mathcal{Z}}
\newcommand{\V}{\mathcal{V}}
\newcommand{\cW}{\mathcal{W}}

\newcommand{\B}{\mathbb{B}}
\newcommand{\K}{\mathbb{K}}
\newcommand{\C}{\mathbb{C}}

\newcommand{\Z}{\mathbb{Z}}
\newcommand{\N}{\mathbb{N}}
\newcommand{\R}{\mathbb{R}}
\newcommand{\unit}{\mathbf{1}}

\newcommand{\g}{\mathfrak{g}}
\renewcommand{\sl}{\mathfrak{sl}}

\newcommand{\SL}{\mathrm{SL}}

\newcommand{\Y}{\mathcal{Y}}

\newcommand{\zem}{\mathfrak{Z}}

\renewcommand{\i}{\mathrm{i}}

\renewcommand{\S}{\mathbb{S}}
\renewcommand{\L}{\mathrm{L}}
\newcommand{\Irr}{\mathrm{V}}
\newcommand{\Verma}{\mathrm{M}}

\newcommand{\Nichols}{H}
\newcommand{\NicholsOf}{\mathfrak{B}}

\usepackage[font={small}]{caption}

\title{Nichols algebras, tensor categories and \\logarithmic Kazhdan-Lusztig correspondences \\
\vspace{.5cm}
A personal tour with the main players.
\author{Simon D. Lentner\\ University of Hamburg }
\date{}
}

\begin{document}

\maketitle

There is a very general picture emerging that conjecturally describes what happens to the representation theory of a vertex algebra $\V$ if we pass to the kernel $\cW$ of a set of screening operators. Namely, the screening operators generate a Nichols algebra $\Nichols$ inside $\Rep(\V)$ and in many cases $\Rep(\cW)$ coincides with the relative Drinfeld center of $\Rep(\Nichols)$. This vastly generalizes the construction of a quantum group as the Drinfeld double of a Nichols algebra over the Cartan part. In this example, the conjectural category equivalence has  been studied since around $20$ years as {\it logarithmic Kazhdan Lusztig correspondence}. 

The present survey was part of my habilitation thesis about my work in this area. I want to make it available as an introductory text, intended for readers from a pure algebra background as well as from a physics background. I motivate and explain gently and informally the different topics involved (quantum groups, Nichols algebras, vertex algebras, braided tensor categories) with a distinct categorical point of view, to the point that I can explain my general expectation. Then I explain some previous results and explain the main techniques in my recent proof of the conjectured category equivalence in case $\V$ is a free field theory and under technical assumptions on the analysis side.    


\vspace{3cm}

\tableofcontents

\newpage

\begin{quote}
As far as I can see, \\
all a-priori statements in physics \\
have their origin in symmetry.\\
(Hermann Weyl 1885-1955)

\bigskip

Romanticizing is like algebraicizing.\\
(Novalis 1772-1801) 
\end{quote}

\section{Symmetries and Lie algebras}

Symmetry is a fundamental experience throughout mathematics and physics.  A more refined idea is: In the presence of symmetry, individual objects can be analyzed by how they \emph{fail} to be fully symmetric and they can be grouped together by symmetry operations. This leads to structural questions on the symmetries as an  abstract group and on its representation theory. 

Historically, this perhaps appears first in Galois theory: Lagrange  in \cite{Lang1770} analyzed the existing solution formulae of the cubic and quartic polynomial equation by defining \emph{resolvents}, expressions that are symmetric under a subset of permutations of all solutions. Ruffini \cite{Ruff1813} and Abel \cite{Ab1826} showed along these lines
that there cannot be a solution of the general $n$-th order polynomial equation by radicals for $n>4$. In 1830 Galois devised a general program of studying algebraic equations in terms of their symmetry groups, in particular the Abel-Ruffini theorem can now be understood to directly reflect the fact that the symmetric group $\S_n$ is for $n>4$ no longer solvable in terms of cyclic~groups.

In \cite{Lie1888} Lie worked with Engel on the ambitious idea of devising a version of Galois theory for differential equations. Perhaps this is best illustrated in the topic of \emph{spherical harmonic functions}, 
which are in a certain sense generalizations of the trigonometric functions to the sphere and play a fundamental role in virtually any analytic problem with spherical symmetry (gravitational or magnetic field of the earth, atomic orbital, etc): Consider the eigenproblem of the spherical Laplace operator
$$\Delta f =c f,$$
where $f:\mathrm{S}^2\to \C$ is an eigenfunction on the sphere and $c\in\C$ the eigenvalue. For example, the reader may be interested in the different ways in which a spherical bell can vibrate and resonate, then a solution $f$ describes a particular spatial pattern of vibration and~$c$ is related to its frequency, so the eigenvalue problem essentially asks for the overtone series. For the sphere, this eigenvalue problem is symmetric under the group of rotations $G=\mathrm{SO}_3(\R)$, and infinitesimally under the Lie algebra $\g=\mathfrak{so}_3$, which acts by certain first order differential operators $\i \L_x,\i \L_y,\i \L_z$ that satisfy the commutator relation $[\L_x,\L_y]=\i\L_z$ and similar relations for $x,y,z$ cyclically permuted. Note that the factor $\i$ becomes just convention when we consider this as a Lie algebra over the complex numbers. The Laplace operator $\Delta$ itself can be expressed as a multiple of the element $\L_x^2+\L_y^2+\L_z^2$ in the universal enveloping algebra $U(\g)$, which is a central element, as it should be for an operator preserved by all symmetries. 

A particular solution $f$ is typically \emph{not} symmetric, so the space of solutions for fixed eigenvalue $c$ has the structure of a representation of $G$ and $\g$, that is, there is a basis of eigenfunctions, such that after rotating an eigenfunction it can be again expressed as a certain linear combination of this basis. Moreover, the product of two eigenfunctions can be expressed in this bases, possibly involving other eigenvalues. The representations and their tensor products are the algebraic content of the problem.

\begin{figure}[h]\label{fig_harmonics}
\begin{center}
\includegraphics[scale=.30]{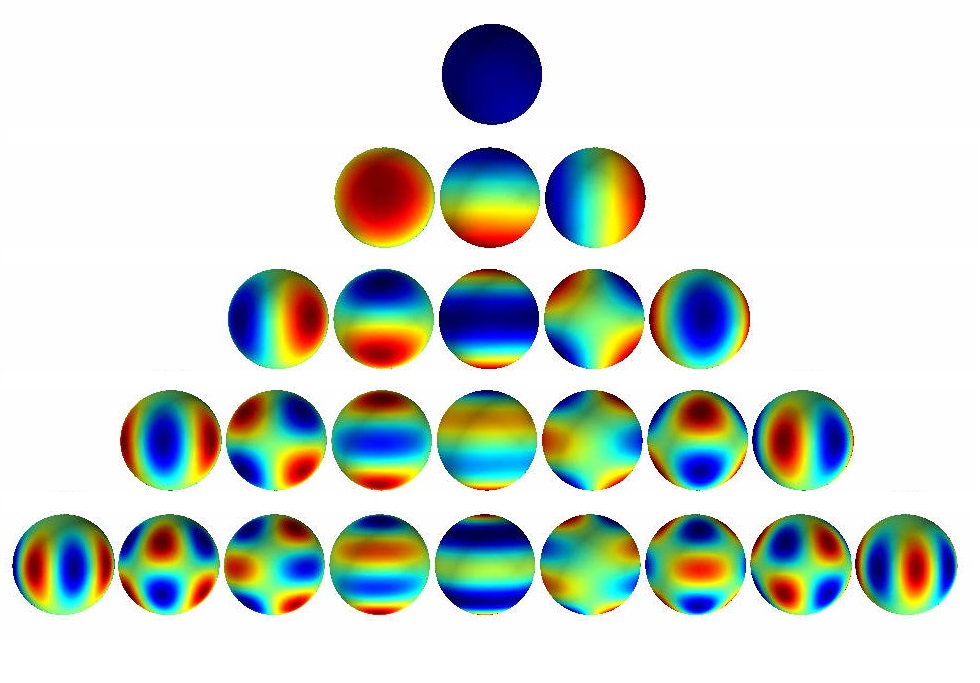}
\caption[From F. Pokorny, KTH Royal Institute of Technology, personal website]{Each row $s=0,1,2,\ldots$ shows a basis of an eigenspace of the spherical Laplace operator, which has eigenvalue $s(s+1)$ and dimension $2s+1$. The horizontal grading is by the $\L_z$-eigenvalue $m=-s,\ldots, s$ and shows solutions with $2|m|$-fold symmetry around the arbitrarily chosen $z$-axis.}
\end{center}
\end{figure}

When we solve this differential equation, 
we empirically find eigenvalues $c=s(s+1)$ for $s=0,1,2,\ldots$ with eigenspaces of every odd dimension $2s+1$, which are irreducible representations of $\mathrm{SO}_3(\R)$. It starts with the 1-dimensional trivial  representation (a symmetric solution, $s$-orbital) and the $3$-dimensional standard representation (a wobble in some single direction, $p$-orbital). The tensor product of two representations again decomposes into such irreducible representations according to the \emph{Clebsch-Gordan rule}:
$$V_s \otimes V_{s'}
=V_{s+s'}\oplus \cdots \oplus V_{|s-s'|}.
$$
Our main point is that important quantities like eigenvalues, dimension and structure of the eigenspaces, and product decompositions can be derived with purely algebraic methods. The symmetry group and its representation theory govern the analytic problem. 

The representations $V$ of a semisimple complex finite-dimensional Lie algebra~$\g$ can be classified and understood as follows: One chooses a commutative subalgebra $\mathfrak{h}\subset \g$ called \emph{Cartan subalgebra}, which is endowed with the nondegenerate \emph{Killing form} $(-,-)$. In our example we choose $\mathfrak{h}$ to be the span of $\L_z$ or $H=2\L_z$. Now for any representation~$V$, where we assume that $\mathfrak{h}$ acts diagonalizably, we may take a simultaneous eigenbasis for~$\mathfrak{h}$ with simultaneous  eigenvalues $\mu \in \mathfrak{h}^*$ called \emph{weights}, that is
$h.v=\mu(h)v$ for all $h\in\mathfrak{h}$. We apply this in particular to the Lie algebra itself with adjoint action $\g_{\mathrm{ad}}$, then the simultaneous eigenvalues are called \emph{roots} $\alpha$. In our example~$\g=\mathfrak{so}_3$ we have in $\g_{\mathrm{ad}}$ the $\L_z$-eigenvectors $E=\L_x+\i \L_y$ and $F=\L_x-\i \L_y$ with $H$-eigenvalues $\pm 2$, or roots $\pm\alpha$ with~$\alpha(H)=2$. In this basis, the commutator relations of $\mathfrak{so}_3$ are as follows:
$$[H,E]=2E,\quad [H,F]=-2F,\quad [E,F]=H.$$
As a remark, in this basis the Lie algebra $\mathfrak{so}_3$ shows to be isomorphic to the Lie algebra $\mathfrak{sl}_2$, over the complex numbers. The action of an eigenvector $x$ in $\g_{\mathrm{ad}}$ with root~$\alpha$ shifts an eigenvector $v$ in $V$ with weight $\mu$ to an eigenvector $xv$ with weight $\mu+\alpha$. In our example, representations are graded by weights $\mu=m\alpha$ and the action by $E,F$ shifts them by~$\pm\alpha$:
\[\begin{tikzcd}[row sep=1ex, column sep=6ex]
\cdots
\arrow[r,swap, shift right=2pt,"E"] 
& 
\arrow[l,swap, shift right=2pt,"F"] 
v'''
\arrow[r,swap, shift right=2pt,"E"] 
& 
\arrow[l,swap, shift right=2pt,"F"] 
v''
\arrow[r,swap, shift right=2pt,"E"] 
& 
\arrow[l,swap, shift right=2pt,"F"] 
v'
\arrow[r,swap, shift right=2pt,"E"] 
& 
\arrow[l,swap, shift right=2pt,"F"] 
\cdots
\\
&
\scriptstyle (m-2)\alpha
&
\scriptstyle (m-1)\alpha
&
\scriptstyle m\alpha
&
\scriptstyle (m+1)\alpha
\end{tikzcd}\]
One can construct general representations called \emph{Verma module} $\Verma_\lambda$ for each weight $\lambda$ and a chosen set of positive roots. In our example, we choose $E$ with eigenvalue $\alpha(H)=+2$, then we define a \emph{highest weight vector} $v$ to be an eigenvector with eigenvalue $\lambda(H)$ for some weight $\lambda=s\alpha$, with the additional property $Ev=0$. Then we induce this representation to a representation of all of $\g$ by adding formal base elements $F^kv$ for all $k\in\N$:
\[\begin{tikzcd}[row sep=1ex, column sep=6ex]
\cdots
\arrow[r,swap, shift right=2pt,"E"] 
& 
\arrow[l,swap, shift right=2pt,"F"] 
F^2v
\arrow[r,swap, shift right=2pt,"E"] 
& 
\arrow[l,swap, shift right=2pt,"F"] 
Fv
\arrow[r,swap, shift right=2pt,"E"] 
& 
\arrow[l,swap, shift right=2pt,"F"] 
v 
&
\hphantom{\cdots}
\\
&
\scriptstyle (s-2)\alpha
&
\scriptstyle (s-1)\alpha
&
\scriptstyle s\alpha
&
\end{tikzcd}\]
This infinite-dimensional representation of $\g$ turns out to be irreducible except if $s\in\frac{1}{2}\N_0$, because in this case the commutator relations show that the arrow $\smash{F^{2s+1}v\stackrel{\scriptstyle E}{\longrightarrow} F^{2s}v}$ happens to be zero. The quotient by the submodule spanned by $F^kv$, $k>2s$ is an irreducible representation $\Irr_\lambda$ of dimension $2s+1$ and the weights   exhibit a nice reflection symmetry:
\[\begin{tikzcd}[row sep=1ex, column sep=6ex]
F^{2s}v
\arrow[r,swap, shift right=2pt,"E"] 
& 
\arrow[l,swap, shift right=2pt,"F"] 
\cdots
\arrow[r,swap, shift right=2pt,"E"] 
& 
\arrow[l,swap, shift right=2pt,"F"] 
v 
\\
\scriptstyle -s\alpha
&
&
\scriptstyle s\alpha
\end{tikzcd}\]
\vspace{-.6cm}

We remark that in physics $s$ is the angular momentum, for example in particle physics the spin, and the Clebsch-Gordan rule describes spin addition for pairs of particles. The $\mathfrak{so}_3$-representations with $s$ half-integer do not lift to representations of $G=\mathrm{SO}_3(\R)$, but to representations of the central extension $G=\mathrm{SU}_2(\C)$, and this corresponds to fermions.

A similar procedure works for any semisimple complex finite-dimensional Lie algebra, and it leads to the classification of such Lie algebras by axiomatized \emph{root systems} $\Phi\subset \mathfrak{h}^*$ and related reflection groups called \emph{Weyl groups}. Root systems in turn are classified in terms of \emph{Dynkin diagrams}, where the vertices depict a basis of \emph{simple roots} $\alpha_1,\ldots,\alpha_n$ in~$\mathfrak{h}$ and the edges express information about the Killing form $(\alpha_i,\alpha_j)$:

\begin{figure}[H]\label{fig_Dynkin}
\begin{center}
\includegraphics[scale=0.7]{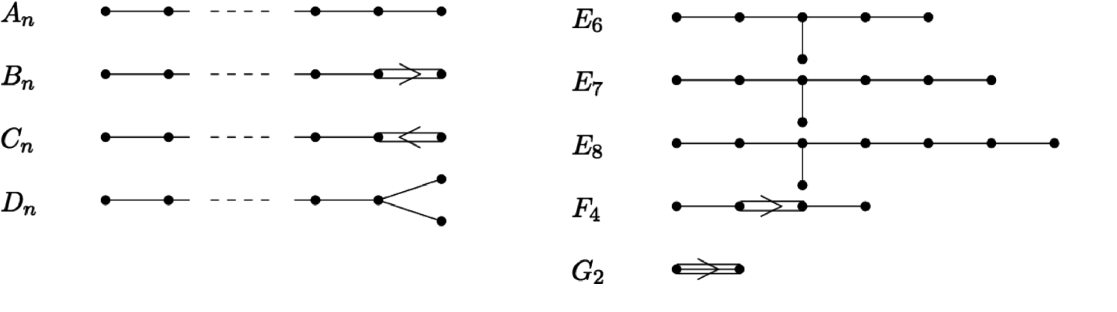}
\caption[Eckes, C., Giardino, V. (2018). The Classificatory Function of Diagrams: Two Examples from Mathematics. In: Diagrammatic Representation and Inference. Lecture Notes in Computer Science, vol 10871, Springer.]{Each Dynkin diagram encodes a root system, a reflection group and a simple Lie algebra. \\ A single node (type $A_1$)  corresponds to our example $\sl_2,\mathfrak{so}_3$ with roots $\pm\alpha$ and Weyl group~$\S_2$. }
\end{center}
\end{figure}

Besides the classification of Lie algebras in terms of root systems, the procedure provides a complete classification of finite-dimensional irreducible representations of $\g$ in terms of integral dominant weights~$\lambda$.  

\begin{figure}[H]
\begin{minipage}{.25\textwidth}
\includegraphics[scale=0.8]{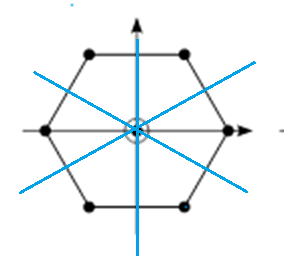}
\end{minipage}
\hspace{.4cm}
\begin{minipage}{.6\textwidth}
\includegraphics[scale=1.7]{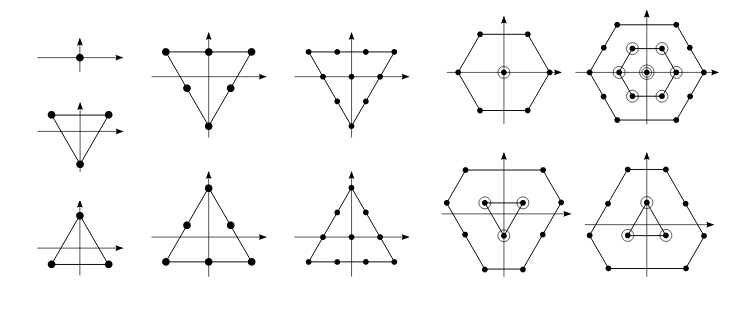}
\end{minipage}
\caption[From Schaeffer, J. (2022). SU(n), Darstellungstheorie und deren Anwendung im Quarkmodell. BestMasters. Springer Spektrum.]{The root system $A_2$ with $3$ positive and $3$ negative roots. Blue lines depict the $3$ orthogonal hyperplanes in $\R^2$, and the reflections on them generate the Weyl group $\S_3$. This encodes the Lie algebras $\mathfrak{sl}_3$ and $\mathfrak{su}_3$, which are isomorphic over $\C$.\\
To the right, the smallest irreducible representations, each point corresponding to a weight. These representations appear e.g. in the physics of strong interaction as part of the standard model and describe different particle multiplets (quarks, antiquarks, gluons, mesons, etc.).
}
\end{figure}

\begin{figure}[H]
\begin{center}
\includegraphics[scale=.25]{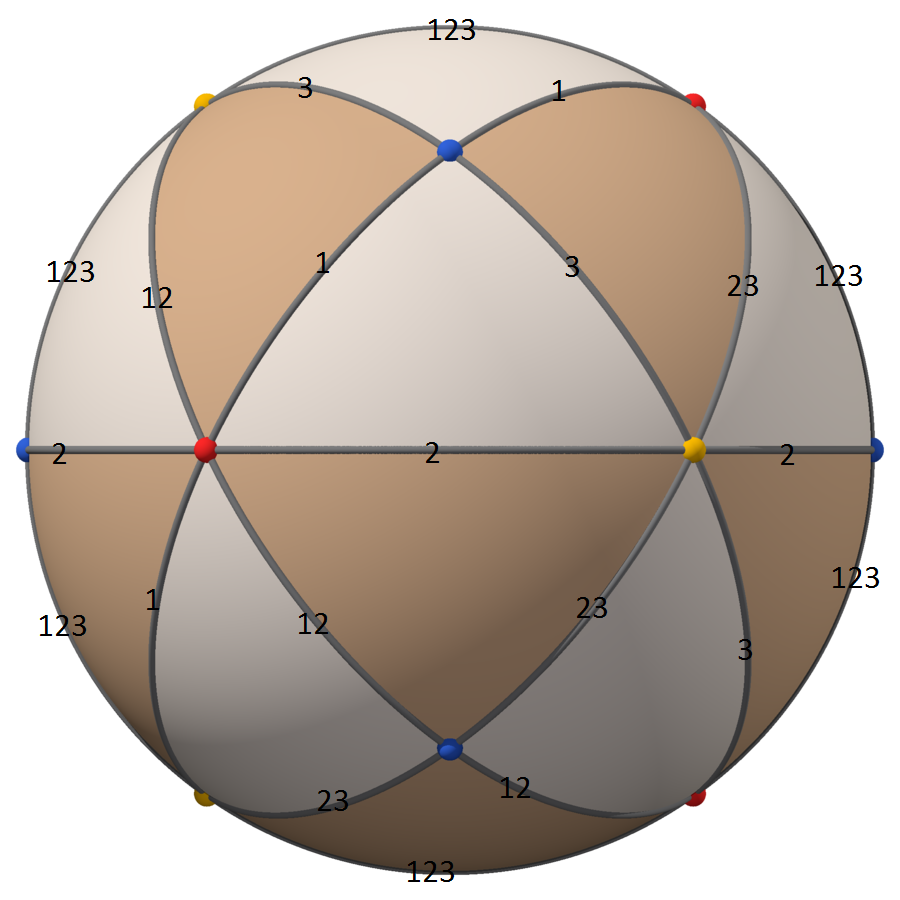}
\caption[From T. Piesk, Wiki Commons, Category: Spherical disdyakis polyhedrapedia]{Root system $A_3$ with $6$ positive roots named $1,2,3,12,23,123$, corresponding to $6$ hyperplanes through the origin in $\R^3$, and $24$ Weyl chambers permuted by the Weyl group $\S_4$. This encodes the Lie algebras $\mathfrak{sl}_4$, $\mathfrak{su}_4$ and $\mathfrak{so}_6$, which are isomorphic over $\C$.}
\end{center}
\end{figure}

We now turn to an example of \emph{quantum field theory}: 
Let $V_{1},\ldots ,V_{n}$ be irreducible representations of $\g$. Let $\Sigma$ be a Riemann surface with distinct punctures $z_1,\ldots,z_n$. Physically, we might imagine in each puncture $z_i$ some particle in the multiplet $V_i$.

\begin{figure}[h]
\begin{center}
\tikzset{pics/.cd,
	handle/.style={code={
			\draw[fill=gray!10]  (-2,0) coordinate (-left) 
			to [out=260, in=60] (-3,-2) 
			to [out=240, in=110] (-3,-4) coordinate (-k)
			to [out=290,in=180] (0,-6) coordinate (-num)
			to [out=0,in=250] (3,-4) 
			to [out=70,in=300] (3,-2) 
			to [out=120,in=280] (2,0)  coordinate (-right);
			\pgfgettransformentries{\tmpa}{\tmpb}{\tmp}{\tmp}{\tmp}{\tmp}
			\pgfmathsetmacro{\myrot}{-atan2(\tmpb,\tmpa)}
			\draw[rotate around={\myrot:(0,-2.5)}] (-1.2,-2.4) to[bend right]  (1.2,-2.4);
			\draw[fill=white,rotate around={\myrot:(0,-2.5)}] (-1,-2.5) to[bend right]coordinate[pos=0.5] (-B) coordinate[pos=1] (-D) (1,-2.5) 
			to[bend right] coordinate [pos=1] (-E) coordinate[pos=0.7] (-A) coordinate[pos=0.4] (-C) (-1,-2.5);
}}}

	\fontsize{8}{8}\selectfont
	\begin{tikzpicture}[scale=0.7]
	\newcommand\shift{0}
	
	\pic[rotate=-90,scale=0.4] (tr) at (0-\shift,0) {handle};
	\pic[rotate=90,scale=0.4] (tl) at (0-\shift,0) {handle};
	\draw[blue,rotate=-270, 
	fill=lightgray] (1+\shift,-2.5)  circle (10pt) node {$V_1$} ;
	\draw[blue,rotate=-270, 
	fill=lightgray] (0+\shift,-3)  circle (10pt) node {$V_2$} ;
\end{tikzpicture}
\normalsize
\end{center}
\caption[Own creation with TikZ]{A Riemann surface of genus $2$ with $2$ punctures.}
\end{figure}
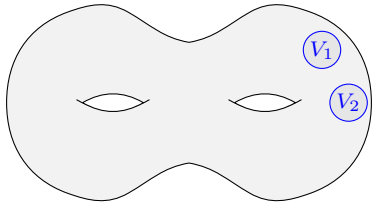

For any choice $\kappa\in\R$ called \emph{shifted level}, and choice of a complex structure on $\Sigma$, we consider 
the \emph{Knishnik-Zomolochikov} differential equation for a  function $\Psi(z_1,\ldots,z_n)$ taking values in $V_1\otimes\cdots\otimes V_n$:
\begin{align}\label{exm_KZequation}
    \left(\kappa\frac{\partial}{\partial z_i} 
 + \sum_{j\neq i}\frac{1}{z_i-z_j}\Omega_{ij}\right)
 \Psi(z_1,\ldots,z_n) =0,
\end{align}
where in our example $\mathfrak{so}_3$ we take $\Omega=\frac12(\L_x\otimes \L_x+\L_y\otimes \L_y+\L_z\otimes \L_z)$
acting on the tensor factors $V_{i}\otimes V_{j}$. Then the space of solutions of this differential equation has a nontrivial monodromy around the singularities $z_i=z_j$, which can be turned into a nontrivial \emph{braiding morphism}
$$V_i\otimes V_j \longrightarrow V_j\otimes V_i,$$
%
%
see \cites{Koh88,Drin89}. Drinfeld defined the \emph{quantum group} $U_q(\g)$ as a deformation of $U(\g)$ by a formal parameter $\smash{q=e^{\frac{\pi\i}{\kappa}}}$. Its representation theory coincides with the representation theory of $\g$, and it also admits a commutative tensor product $V\otimes W$, but the braiding morphism is nontrivial.
This leads to the axiomatization of a \emph{braided tensor category}, a representation theory with a commutaive tensor product and good properties, even if there is no underlying (symmetry-) group anymore. 

Lusztig, who had previously made significant contributions to the representation theory of Lie groups in finite characteristic \cite{Lusz84}, realized that if the quantum group parameter $q$ is specialized to a root of unity, then the representation theory is not semisimple anymore and it mirrors many features of the finite characteristic setup. He defined a finite-dimensional quotient called \emph{small quantum group} $u_q(\g)$, whose representation theory is related to the representation theory of $\g$ over a field of finite characteristic.

\begin{quote}
In my work, I construct, study and classify quantum groups and braided tensor categories. Based on this, I construct new conformal quantum field theories that involve such representation theories.
\end{quote}  

\section{Quantum groups and generalizations from a categorical perspective}\label{sec_quantumgroups}

Let me first sketch the construction of a quantum group, from my present more categorical perspective. This perspective is for me a consequence of being exposed as a student in Munich by H.-J. Schneider and Y. Sommerhäuser to the classification program of Hopf algebras, and as a postdoc in Hamburg by C. Schweigert and I.~Runkel to the language of tensor categories. Effectively, this perspective  produces a vast amount of new \emph{generalized quantum groups} in the sense of nonsemisimple (modular) braided tensor categories obtained from a given semisimple (modular) braided tensor category, as I discuss now.

\bigskip

Recall the definition of a braided tensor category \cite{EGNO15}: It is a \emph{category} $\cC$ consisting of objects $X,Y$ etc. and sets of morphisms $\Hom(X,Y)$, which are $\K$-vector spaces. We assume all the typical features, such as kernels, direct sums, etc. In such  a category, we are interested in the simple objects, indecomposable objects and projective objects. The reader may have in mind representations of a finite-dimensional nonsemisimple algebra. A \emph{tensor category} is endowed with a tensor product of objects $X\otimes Y$ and morphisms, and with a choice of a natural isomorphism implementing associativity
$$a_{X,Y,Z}:\;(X\otimes Y)\otimes Z \stackrel{\sim}{\longrightarrow} X\otimes (Y\otimes Z).$$
A \emph{braided tensor category} $\cC$ is moreover endowed with a choice of a natural isomorphism implementing commutativity of the tensor product 
$$c_{X,Y}:\;X\otimes Y \stackrel{\sim}{\longrightarrow} Y\otimes X$$
that fulfills the relations of the braid group. If $c_{Y,X}c_{X,Y}=\id$, then the braiding is called \emph{symmetric}, since it then fulfills the relations of the symmetric group. We are interested in cases where this is maximally not the case, then the braiding is called \emph{nondegenerate}. 

\begin{example}
    For a group $\Gamma$, the category $\Vect_\Gamma$ of $\Gamma$-graded $\K$-vector spaces has simple objects $\K_a,\,a\in\Gamma$, the one-dimensional vector space of degree $a$, and any object is a direct sum of these, so the category is semisimple.

    Similar to the category of vector spaces, this becomes a tensor category with the trivial associator, and if $\Gamma$ is abelian with the trivial  braiding. But we may also modify them by suitable choices of scalars $\omega:\Gamma\times\Gamma\times\Gamma\to \K^\times$ and $\sigma:\Gamma\times\Gamma\to \K^\times$ as follows:
$$
\begin{tikzcd}
(\K_a\otimes \K_b)\otimes \K_c
\arrow[rr, dashed,"a_{\K_a,\K_b,\K_c}"]
\arrow[d,equal]
&&
\K_a\otimes (\K_b\otimes \K_c)
\arrow[d,equal]
\\
\K_{a+b+c}
\arrow{rr}{\omega(a,b,c)}
&&
\K_{a+b+c}
\end{tikzcd}
\begin{tikzcd}
\K_a\otimes \K_b
\arrow[rr, dashed,"c_{\K_a,\K_b}"]
\arrow[d,equal]
&&
\K_b\otimes \K_a
\arrow[d,equal]
\\
\K_{a+b}
\arrow{rr}{\sigma(a,b)}
&&
\K_{b+a}
\end{tikzcd}
$$
The choices of scalars $(\omega,\sigma)$ up to the natural equivalence relation are classified by quadratic forms on the abelian group $\Gamma$, see \cites{MacL52,JS93}. The braiding is nondegenerate iff the quadratic form is nondegenerate. 
\end{example}

A source of tensor categories $\cB$ are (categorical) Hopf algebras: Let $\Nichols$ be an algebra in a tensor category $\cC$. There is a straightforward definition of a $\Nichols$-representation on an object $M$ inside $\cC$, 
$$\rho:\Nichols\otimes M\to M,$$
and accordingly there is an abelian category of representations of $H$ inside $\cC$
$$\cB=\Rep(\Nichols)(\cC).$$ 
\begin{example}
    An algebra $\Nichols$ in the category of $\Gamma$-graded vector spaces $\cC=\Vect_\Gamma$ is a $\Gamma$-graded algebra. A module over $\Nichols$ in $\cC$ is a $\Gamma$-graded vector spaces with an action of $\Nichols$, such that the action is compatible with the $\Gamma$-gradings of $\Nichols$ and $M$
\end{example} 
There are two different ideas to construct a tensor product in $\Rep(\Nichols)$: We may want to take a tensor product $M\otimes_H N$ over $H$, provided $\Nichols$ is a commutative algebra, this is discussed in Section~\ref{sec_quantumgroups}. We may alternatively want to take the tensor product  $M\otimes N$ in the underlying category $\cC$, then we need some rule how $\Nichols$ can act on two factors. A \emph{bialgebra} is an algebra together with the morphisms in the tensor category
$$\Delta:\Nichols\to \Nichols\otimes\Nichols,\qquad 
\varepsilon:\Nichols\to \unit.$$
The axioms are in such a way that  $\varepsilon$ produces an action on the tensor unit $\unit$ and $\Delta$ produces an action on $M\otimes N$ via 
$$\Nichols\otimes (M\otimes N)
\stackrel{\Delta}{\longrightarrow}
(\Nichols\otimes \Nichols)\otimes (M\otimes N)
\stackrel{c_{\Nichols,M}}{\longrightarrow}
(\Nichols\otimes M)\otimes (\Nichols\otimes N)
\stackrel{\rho \otimes \rho}{\longrightarrow}
M\otimes N.
$$
Note that we need $\cC$ to be braided for this definition, or alternatively to be able to formulate how $H\otimes H$ is an algebra in $\cC$. For some early appearances of this ides, see \cites{Ros78, Rad85, Maj95, Ber95}.

 A \emph{Hopf algebra} has an additional antialgebra morphism $S:\Nichols\to\Nichols$ that produces an action on the dual object $M^*$ in the sense of rigidity. A \emph{quasi-triangular Hopf algebra} has an additional element $R\in H\otimes H$ that produces a braiding for $H$-modules. Good references for Hopf algebras in vector spaces are \cites{Mon93, Kas97, Schn95}.

\begin{example}
Let $\cC$ be the category of vector spaces with the trivial braiding. Two very familiar examples of Hopf algebras, which motivate much of the theory, are as follows:
\begin{itemize}
\item Let $G$ be a group, then the group algebra $\K G$ is the vector space spanned of $G$, with the multiplication inherited from $G$. It becomes a Hopf algebra with the structures
$$\Delta(g)=g\otimes g,\qquad \varepsilon(g)=1,\qquad S(g)=g^{-1}.$$
This additional structures describes how a group acts usually on the tensor product, the trivial representation and the dual representation.
\item Let $\g$ be a Lie algebra, then the universal enveloping algebra $U(\g)$ is the tensor algebra of $\g$ (i.e. the free algebra in a basis of $\g$) modulo relations $xy-yx=[x,y]_\g$. It becomes a Hopf algebra with the structures
$$\Delta(x)=x\otimes 1+1\otimes x,\qquad \varepsilon(x)=0,\qquad S(x)=-x.$$
This additional structures describes how a Lie algebra, for example an algebra of derivations, acts usually on the tensor product, the trivial representation and the dual representation.
\end{itemize}
Both tensor categories admit a trivial braiding, because $\Delta$ is cocommutative.
\end{example}
If $\cC$ itself is the braided tensor category of representations of a quasitriangular Hopf algebra $L$, then $\Rep(\Nichols)(\cC)$ consists of the representations over the Radford biproduct or smash product $\Nichols\# L$, which I like to denote analogously to the semidirect product of groups $\Nichols\rtimes L$. Note that the setup is slightly different then in literature, where often $L$-Yetter-Drinfeld modules are used instead of assuming $L$ to be quasitriangular. 

Conversely, Tannaka-Krein reconstruction, which was originally devised for algebraic groups, states in the version in \cite{Schau91} that any tensor category $\cB$ with a faithful exact tensor functor $F:\cB\to \Vect$ is equivalent to the category of representations of a Hopf algebra $\Nichols\in\Vect$ and $F$ is the functor forgetting the action. In \cite{LM24} we give the following relative version for fibre functors $F:\cB\to \cC$.  In a Hopf algebra setting it reduces to a version of the Radford projection theorem \cite{Rad85}, which in turn generalizes the fact that a split exact sequence of groups leads to a semidirect product.

\begin{theorem}\label{thm_LM}
Let $\cB\supset \cC$ be finite rigid monoidal categories and let $\cC$ be a central braided subcategory. Then the existence of a faithful exact tensor functor $F:\cB\to \cC$, which is the identity on $\cC$, is equivalent to the existence of a Hopf algebra $\Nichols$ in $\cC$ such that 
$$\cB=\Rep(\Nichols)(\cC)$$
\end{theorem}

A very interesting source for Hopf algebras in a braided tensor category $\cC$ is the \emph{Nichols algebra} $\NicholsOf(X)$, which is attached to any choice of an object $X\in\cC$ as follows: We can always turn the tensor algebra of $X$ into a Hopf algebra by setting $\Delta(x)=1\otimes x+x \otimes 1$, in which case $x$ is called a \emph{derivation} or \emph{primitive element}. Then $\Delta$ of formal products of such elements is uniquely fixed by multiplicativity and this depends on the choice of the braiding. This may lead to surprises, for example a power  of  a primitive element may happen to be primitive itself. The Nichols algebra $\NicholsOf(X)$ is the quotient by the largest Hopf ideal intersecting trivially with $X$. In particular, it contains no primitive elements other then $X$, and this turns out to be an equivalent characterization.
\begin{example}\label{exm_NicholsRank1}
Let us take $\cC=\Vect_\Gamma^{\omega,\sigma}$, and let $X=\K_a$ be the one-dimensional vector space in degree $a\in \Gamma$ with basis $x$. We compute 
$$\Delta(x^2)=\Delta(x)^2=(1\otimes x+x\otimes 1)^2
=x^2\otimes 1+(x\otimes x)+\sigma(a,a)(x\otimes x)+x^2\otimes 1.$$
In particular if $\sigma(a,a)=-1$, then $x^2$ is a again primitive. In this case, the relation $x^2=0$ is compatible with the Hopf algebra structure and thus becomes a relation in the Nichols algebra. Similarly, if $\sigma(a,a)$ is a primitive $\ell$-th root of unity, then the Nichols algebra is, at least in characteristic zero  
$$\NicholsOf(X)=\K[x]/(x^\ell).$$
In all other cases, the Nichols algebra is the free polynomial ring in one variable $\K[x]$.

Let us mention that the same happens for a trivial braiding over a field of characteristic $\ell$, that is, the $\ell$-th power of a derivation is again a derivation, because the binomial coefficients ${\ell\choose k}$ vanish for $0<k<\ell$. This lead to the notion of $\ell$-restricted Lie algebras, and after this Lusztig modeled the braided case in characteristic zero presented above.   
\end{example}
An equivalent definition of Nichols algebras can be given in terms of the quantum symmetrizer map, which roughly measures to which extend the action of the braid group $\mathbb{B}_n$ on $X^{\otimes n}$ does not factor to an action of the symmetric group $\mathbb{S}_n$, see \cite{Len21}~Section~2. 

Nichols algebras  appear naturally in every Hopf algebra, as the graded algebra generated by the primitive elements. The Andruskiewitsch-Schneider program \cite{AS10}, for a current state see the survey \cite{AG19}, uses this to aim at a classification of  all Hopf algebras with a fixed maximal cosemisimple part $C$. This was successful for example in the case where $C$ is an abelian group ring, and it turns out that in this case the Hopf algebra is completely determined by $C$, the Nichols algebra and certain deformations, by work of Andruskiewitsch, Schneider, Angiono, Garcia Iglesia and many others, see for example \cites{An13,AG11}.

In these cases, the braiding of $X$ is of \emph{diagonal type}, that is, there is a distinguished basis $x_i$ and the braiding is $x_i\otimes x_j\to q_{ij}(x_j\otimes x_i)$, here $q_{ij}=\sigma(\alpha_i,\alpha_j)$. Following Heckenberger, we depict such a braiding by writing a $q$-diagram with nodes for each $\alpha_i$, decorated by the self braiding $q_{ii}$, and edges from $\alpha_i$ to $\alpha_j$, decorated by the double braiding $q_{ij}q_{ji}$, and the edge is only drawn if $q_{ij}q_{ji}\neq 1$:
\begin{center}
\begin{tikzpicture}
		\draw (0,0)--(1.6,0);
		\draw (-0.1,0) circle[radius=0.1cm] node[anchor=south]{$ q_{11}$}
		(1.7,0) circle[radius=0.1cm] node[anchor=south]{$ q_{22}$};
		\draw (0.8,0) node[anchor=south]{$ q_{12}q_{21}$};
\end{tikzpicture} 
\end{center}
This does not fix the braiding completely, but it turns out to contain all relevant information. From the categorical view on $\Vect_\Gamma^{\omega,\sigma}$, it encodes the quadratic form and the associated bimultiplicative form, which is precisely the information invariant under equivalences of braided tensor categories.

The case of a quantum group was the main motivating example, and in this context Nichols algebras  already appear in Lusztig's work \cite{Lusz93} Chapter 1 as the algebra~$\mathfrak{f}$.

\begin{example}[Quantum Borel part]\label{exm_quantumBorelpart}
    Let $\g$ be a  finite-dimensional semisimple complex Lie algebra, with a choice of simple roots $\alpha_1,\cdots,\alpha_n$ and Killing form $(\alpha_i,\alpha_j)$. Let $\cC=\Vect_\Gamma^{1,\sigma}$ for $\Gamma=\Z^n$ with formal basis $\alpha_i$ and braiding $\sigma(\alpha_i,\alpha_j)=q^{(\alpha_i,\alpha_j)}$ for some $q\in\K^\times$. For example, we depict the $q$-diagrams for the cases $A_2$ and $B_2$:
\begin{center}
\begin{tikzpicture}
		\draw (0,0)--(1.6,0);
		\draw (-0.1,0) circle[radius=0.1cm] node[anchor=south]{$ q^2$}
		(1.7,0) circle[radius=0.1cm] node[anchor=south]{$ q^{2}$};
		\draw (0.9,0) node[anchor=south]{$ q^{-2}$};
\end{tikzpicture} 
\hspace{2cm}
\begin{tikzpicture}
		\draw (0,0)--(1.6,0);
		\draw (-0.1,0) circle[radius=0.1cm] node[anchor=south]{$ q^2$}
		(1.7,0) circle[radius=0.1cm] node[anchor=south]{$ q^4$};
		\draw (0.9,0) node[anchor=south]{$ q^{-4}$};
\end{tikzpicture} 
\end{center}

    Then the Nichols algebra $\NicholsOf(X)$ of the object $X=\bigoplus_i \K_{\alpha_i}$ is isomorphic to the positive part of the quantum group $U_q(\g)^+$ resp. the small quantum group $u_q(\g)^+$ if $q$ is a root of unity. That is, the  $q$-deformed Serre relations and the truncation relations for root vectors, which are both quote complicated, follow automatically and uniquely from the braiding on the $q$-deformed Cartan part and the necessities of having a Hopf algebra structure.  

    Altogether, the representations of $\NicholsOf(X)$ over $\Vect_\Gamma^{1,\sigma}$ for some quotient $\Gamma$ of $\Z^n$ correspond to representations of the Borel part of the quantum group $U_q(\g)^{\geq 0}$ resp. $u_q(\g)^{\geq 0}$.

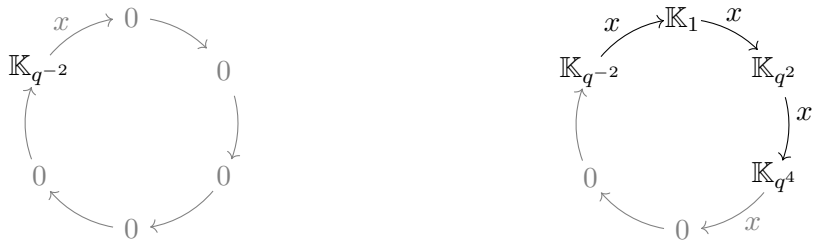
\begin{figure}[h]
\newcommand{\radiuscircle}{2cm}

\noindent
\begin{minipage}{.5\textwidth}
\centering
\begin{tikzpicture}[->,scale=.7]
   \node[gray] at (125:2.3cm) {$x$};

   \node[gray]  at (210:\radiuscircle) {$0$};
   \node at (150:\radiuscircle) {$\K_{q^{-2}}$};
   \node[gray] at (90:\radiuscircle)  {$0$};
   \node[gray] at (30:\radiuscircle) {$0$};
   \node[gray]  at (-30:\radiuscircle) {$0$};
   \node[gray]  at (-90:\radiuscircle) {$0$};

   \draw[gray] (200:\radiuscircle)  arc (200:160:\radiuscircle);
   \draw[gray] (140:\radiuscircle)  arc (140:100:\radiuscircle);
   \draw[gray] (80:\radiuscircle)  arc (80:45:\radiuscircle);
   \draw[gray] (15:\radiuscircle)  arc (15:-20:\radiuscircle);
   \draw[gray] (-40:\radiuscircle)  arc (-40:-80:\radiuscircle);
   \draw[gray] (-100:\radiuscircle)  arc (-100:-140:\radiuscircle);
\end{tikzpicture}
\end{minipage}%
\begin{minipage}{.5\textwidth}
\centering
\begin{tikzpicture}[->,scale=.7]
   \node at (125:2.3cm) {$x$};
    \node at (65:2.3cm) {$x$};
        \node at (5:2.3cm) {$x$};
                \node[gray] at (-55:2.3cm) {$x$};

   \node[gray]  at (210:\radiuscircle) {$0$};
   \node at (150:\radiuscircle) {$\K_{q^{-2}}$};
   \node at (90:\radiuscircle)  {$\K_{1}$};
   \node at (30:\radiuscircle) {$\K_{q^{2}}$};
   \node at (-30:\radiuscircle) {$\K_{q^4}$};
   \node[gray]  at (-90:\radiuscircle) {$0$};

   \draw[gray] (200:\radiuscircle)  arc (200:160:\radiuscircle);
   \draw (140:\radiuscircle)  arc (140:100:\radiuscircle);
   \draw (80:\radiuscircle)  arc (80:45:\radiuscircle);
   \draw (15:\radiuscircle)  arc (15:-20:\radiuscircle);
   \draw[gray] (-40:\radiuscircle)  arc (-40:-80:\radiuscircle);
   \draw[gray] (-100:\radiuscircle)  arc (-100:-140:\radiuscircle);
\end{tikzpicture}
\end{minipage}
\caption[Own creation with TikZ]{A simple and an indecomposable module over $\K[x]/x^p$ in $\Vect_\Gamma$, or equivalently $u_q(\sl_2)^{\geq0}$.}
\end{figure}
\end{example}

Heckenberger has in \cite{Heck09} classified the finite-dimensional Nichols algebras, over $C$ a finite abelian group ring. The reader is referred to the book \cite{HS20}, and for a detailed list of generators and relations to the survey \cite{AA17}. His key insight was that any such Nichols algebra comes with a generalized root system \cite{HY08}, which was then established in a much broader setting in \cite{AHS10}. A generalized root system can be defined as a set of hyperplanes through the origin, and all integrality conditions of Lie algebra root systems stay in place, but it is generalized in the sense that there is no underlying scalar product, so Weyl reflections are merely linear involutions and in particular they may connect differently shaped Weyl chambers. This includes the root systems of Lie superalgebras, where the occurrence of different Weyl chambers is already familiar. It is surprising that this definition is still very rigid: The finite generalized root systems have been classified in \cite{CH15}. Compared to the Lie algebra case there is only one additional series $D(n|m)$, familiar from Lie superalgebras, and $74$  exceptional root systems including $G_2,F_4,E_6,E_7,E_8$. For more details on generalized root systems, in particular the example in Figure \ref{fig_D21}, the reader is referred to \cite{FL22} Section 2.

\begin{figure}
	\begin{center}
		\includegraphics[scale=.3]{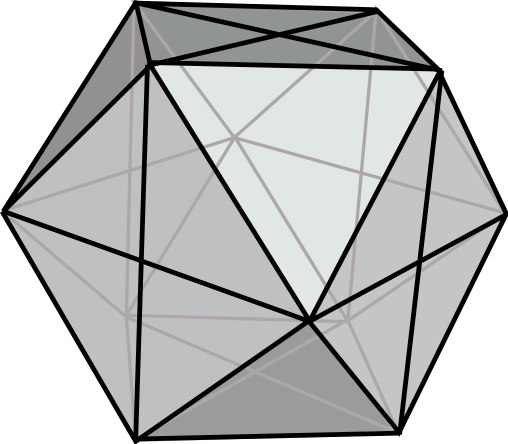}
		\hspace{1cm}
		\includegraphics[scale=.12]{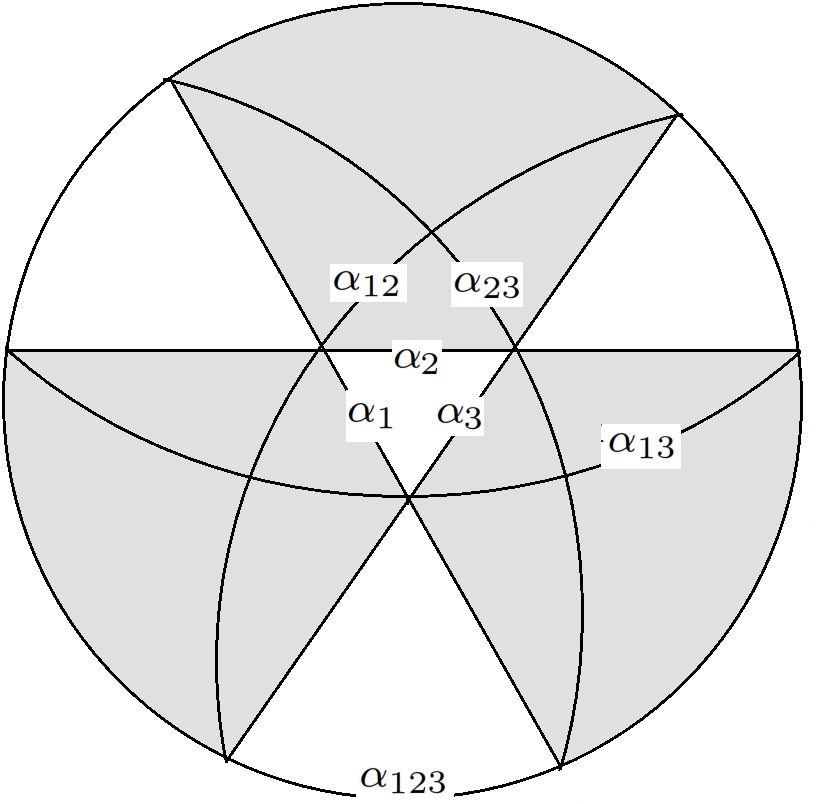}
	\end{center}
    \caption[Own creation with M. Cuntz, University of Hannover]{The root system $D(2,1)$ with $7$ positive roots and 
	$2$ different types of 	Weyl chamber}
    \label{fig_D21}
    \end{figure}

\bigskip

In \cites{Len12, Len14} I constructed large families of  examples of Nichols algebras over nonabelian groups, each from a given Nichols algebra over an abelian group in Heckenberger's list together with a diagram automorphism. For $A_2\cup A_2$ with diagram automorphism $\Z_2$, this reduces to the first example of a Nichols algebra over the dihedral group in \cite{MS00}.  The construction uses the physical idea of orbifolding and this connection will reappear below. Subsequently, Heckenberger and Vendramin find as part of their celebrated classification result \cite{HV17} that my examples indeed exhaust all families of finite dimensional Nichols algebras over nonabelian group except some exceptional cases in rank $2$ and $3$ (note that rank $1$ is still open).

\bigskip

In \cite{ALS23}, which is joint work with I. Angiono and G. Sanmarco, we used this to complete the Andruskiewitsch-Schneider program for nonabelian groups, in the sense that we constructed and classified all finite-dimensional Hopf algebras containing a nonabelian group ring as coradical. The main idea is that we can prove, by rather tedious group cohomology computations, that every Nichols algebra is twist equivalent to an orbifold above. Once this is achieved, the orbifold construction provides a superior alternative set of generators and relations, from the underlying Nichols algebra over the abelian group, and we can use established technology of my coauthors from the abelian group case to compute deformations (called liftings) and prove generation in degree one.

\bigskip

We now continue towards the categorical construction of the quantum group and its generalized versions: Given a finite tensor category $\cB$, there is a standard construction called \emph{Drinfeld center} $\cZ(\cB)$, which by definition consists of pairs of objects $X\in\cB$ and half-braidings $b_{X,Y}:X\otimes Y\to X\otimes Y$ for all $Y\in \cB$, such that in the second argument the hexagon identity holds. This produces a tensor category with a nondegenerate braiding. If~$\cB$ has a central subcategory $\cC$, that is, there are half-braidings $c_{X,Y}:X\otimes Y\to X\otimes Y$ for all $X\in\cB$, $Y\in\cC$, and in particular $\cC$ itself is braided, then there is the notion of a \emph{relative Drinfeld center} 
$\cZ_\cC(\cB)$, see \cites{Maj95,LW22}. It consisting of pairs of objects $X\in\cB$ and half-braidings $b_{X,Y}:X\otimes Y\to X\otimes Y$ for all $Y\in \cB$ such that $b_{X,Y}=c_{X,Y}$ for all $Y\in\cC$. In particular, it contains $\cC$ as a braided tensor category, but not the full center of $\cC$. If $\cB=\Rep(\Nichols)$ for a Hopf algebra $\Nichols\in\cC$, then the relative Drinfeld center can be written more explicitly as a category of $H$-Yetter-Drinfeld modules inside $\cC$ as defined in \cite{Besp95}. 
The adjoint of the functor forgetting the half-braiding gives an induction functor from $\cB$ and in particular from $\cC$ to the relative center. As surely intended  by most of the authors mentioned in this section, we have:
\begin{theorem}\label{thm_quantumgroupscat}
For the Nichols algebra $\Nichols=\NicholsOf(X)$ in $\Vect_\Gamma^{1,\sigma}$ in Example \ref{exm_quantumBorelpart}, we have an equivalence of nonsemisimple (modular) rigid braided tensor categories
$$\cZ_\cC(\Rep(\Nichols))\cong {^H_H}\mathcal{YD}(\cC) \cong \Rep(U_q(\g))\;\text{ resp. }\;\Rep(u_q(\g))$$
The induction functors mentioned above correspond in this case to the familiar induction from the Borel resp. Cartan part.
\end{theorem}

\newcommand{\radiuscircle}{2cm}
\newcommand{\radiuscircleB}{2.4cm}
\newcommand{\radiuscircleA}{2.2cm}

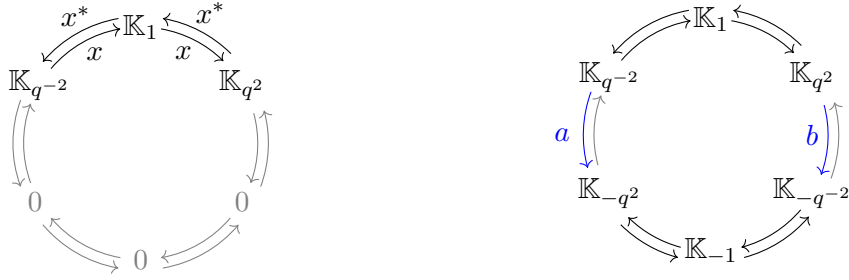
\begin{figure}
\begin{center}
\begin{minipage}{.5\textwidth}
\centering
\begin{tikzpicture}[->,scale=.7]
    \node at (117:2.7cm) {$x^*$};
    \node at (62:2.8cm) {$x^*$};
     \node at (117:1.9cm) {$x$};
    \node at (65:1.9cm) {$x$};

   \node[gray]  at (210:\radiuscircleA) {$0\;$};
   \node at (150:\radiuscircleA) {$\K_{q^{-2}}$};
   \node at (90:\radiuscircleA)  {$\K_{1}$};
   \node at (30:\radiuscircleA) {$\K_{q^{2}}$};
   \node[gray] at (-30:\radiuscircleA) {$0$};
   \node[gray]  at (-90:\radiuscircleA) {$0$};

   \draw[gray] (200:\radiuscircleA)  arc (200:160:\radiuscircleA);
    \draw[gray] (160:\radiuscircleB)  arc (160:200:\radiuscircleB);
   \draw (140:\radiuscircleA)  arc (140:100:\radiuscircleA);
    \draw (100:\radiuscircleB)  arc (100:140:\radiuscircleB);
   \draw (80:\radiuscircleA)  arc (80:45:\radiuscircleA);
    \draw (45:\radiuscircleB)  arc (45:80:\radiuscircleB);
   \draw[gray] (15:\radiuscircleA)  arc (15:-20:\radiuscircleA);
    \draw[gray] (-20:\radiuscircleB)  arc (-20:15:\radiuscircleB);
   \draw[gray] (-40:\radiuscircleA)  arc (-40:-80:\radiuscircleA);
      \draw[gray] (-80:\radiuscircleB)  arc (-80:-40:\radiuscircleB);
   \draw[gray] (-100:\radiuscircleA)  arc (-100:-140:\radiuscircleA);
      \draw[gray] (-140:\radiuscircleB)  arc (-140:-100:\radiuscircleB);

\end{tikzpicture}
\end{minipage}%
\begin{minipage}{.5\textwidth}
\centering
\begin{tikzpicture}[->,scale=.7]
    %
    \node[blue] at (180:2.8cm) {$a$};
     \node[blue] at (0:1.9cm) {$b$};

   \node  at (210:\radiuscircleA) {$\K_{-q^{2}}$};
   \node at (150:\radiuscircleA) {$\K_{q^{-2}}$};
   \node at (90:\radiuscircleA)  {$\K_{1}$};
   \node at (30:\radiuscircleA) {$\K_{q^{2}}$};
   \node at (-30:\radiuscircleA) {$\K_{-q^{-2}}$};
   \node  at (-90:\radiuscircleA) {$\K_{-1}$};

   \draw[gray] (195:\radiuscircleA)  arc (195:160:\radiuscircleA);
    \draw[blue] (160:\radiuscircleB)  arc (160:195:\radiuscircleB);
   \draw (140:\radiuscircleA)  arc (140:100:\radiuscircleA);
    \draw (100:\radiuscircleB)  arc (100:140:\radiuscircleB);
   \draw (80:\radiuscircleA)  arc (80:45:\radiuscircleA);
    \draw (45:\radiuscircleB)  arc (45:80:\radiuscircleB);
   \draw[blue] (15:\radiuscircleA)  arc (15:-20:\radiuscircleA);
    \draw[gray] (-20:\radiuscircleB)  arc (-20:15:\radiuscircleB);
   \draw (-40:\radiuscircleA)  arc (-40:-75:\radiuscircleA);
      \draw (-75:\radiuscircleB)  arc (-75:-40:\radiuscircleB);
   \draw (-105:\radiuscircleA)  arc (-105:-135:\radiuscircleA);
      \draw (-135:\radiuscircleB)  arc (-135:-105:\radiuscircleB);

\end{tikzpicture}
\end{minipage}%
\end{center}
\caption[Own creation with TikZ]{A simple Yetter-Drinfeld module over $\K[x]/x^p$ in $\Vect_\Gamma$, corresponding to the standard representation of $u_q(\sl_2)$, and the indecomposable extensions with parameters $(a:b)$ for $q^{6}=-1$.}
\end{figure}

The reader interested in details and an explicit translation to quantum groups in the general setup may consult \cite{CLR23} Section 6. One of the powers of this approach is that it covers several variants of quantum groups appearing in literature, typically by slightly different choices of $\cC$. For example, the \emph{simply connected quantum group} corresponds to choosing $\Gamma$ to be a quotient of the weight lattice instead of  the root lattice, and the \emph{unrolled quantum group} in \cite{CGP15} corresponds to the choice of $\Rep(\mathfrak{h})$ for the abelian Cartan algebra $\mathfrak{h}\subset \g$ instead of representations $\Vect_\Gamma$ of the exponentiated Cartan group. 

\enlargethispage{1cm}
More severely, if $q$ has an even order (which is relevant in the quantum field theory application below), then it was a notorious issue that $u_q(\g)$ does not always admit a braiding, let alone a nondegenerate braiding, see \cite{KS11}.  In \cite{CGR20} this was resolved by introducing a quasi-Hopf algebra version $\tilde{u}_q(\sl_2)$. The approach now presented in Theorem~\ref{thm_quantumgroupscat} produces without further work a  braided tensor category $\Rep(\tilde{u}_q(\g))$, based on a Nichols algebra over a Cartan part $\Vect_\Gamma^{\omega,\sigma}$ with an associator $\omega$, which is for an even group $\Gamma$ necessary in order to have a nondegenerate braiding on $\Vect_\Gamma$. In \cite{GLO18} we have previously worked out a corresponding quasi-Hopf algebra $\tilde{u}_q(\g)$ for arbitrary $\g$. In \cites{AG03, Neg21} a very different construction is done coming from Lusztig's quantum group of divided powers.
\footnote{Technically, it is nowhere spelled out that these three approaches give equivalent results. This would be a nice thesis project.}



\section{Vertex algebras and the logarithmic Kazhdan-Lusztig conjecture}

A \emph{vertex operator algebra} $\V$ is, very roughly speaking, a commutative algebra depending analytically on a coordinate $z$ in the punctured disc $\C^\times$. More precisely, it is a graded vector space with a multiplication map
\begin{align*}
    \Y:\; \V\otimes \V &\to \V[[z,z^{-1}]] \\
    a\otimes b &\mapsto \Y(a,z)b 
\end{align*}
 taking values in Laurent series in a variable $z$ with coefficients in $\V$. The axioms of a vertex operator algebra include a version of commutativity or locality, which relates~$\Y(a,z)\Y(b,w)$ and $\Y(b,w)\Y(a,z)$ for $z,\;w,\;z-w\neq 0$. As an implication, one also has a version of associativity, which relates these two expressions to $\Y(\Y(a,z-w)b,w)$. 
An additional axiom requires that conformal transformations of the variable $z$ in $\Y(a,z)$ are compatible with an action of the Virasoro algebra on $\V$, which is part of the data. Standard mathematical textbooks on vertex operator algebras include \cites{Kac97,FBZ04}. I~was exposed to these topics by M. Schottenloher in Munich \cite{Schot08}, who supervised my diploma thesis on the construction of vertex algebras from Hopf algebra data \cite{Len07}, as well as my PhD thesis.

Similarly, there is a notion of a \emph{vertex algebra module}
$$\Y_\cM:\; \V\otimes \V \to \V[[z,z^{-1}]],$$
and accordingly a category $\Rep(\V)$ consisting of a suitable type of modules (e.g. $C_1$-cofinite). Moreover, for any three modules $\cM,\cN,\cL$ there is the notion of an \emph{intertwining operator} 
$$\Y_{\cL \choose \cM\;\cN}:\; \cM\otimes \cN \to \cL\{z\}[\log(z)],$$
which now depends on $z$ as a multivalued function with a regular singularity (hence it can be expanded in terms of arbitrary complex powers of $z$ and logarithms). Algebraically, these intertwining operators should be viewed as analogs of $V$-balanced bilinear maps between modules $M,N,L$ over a commutative ring $V$. As suggested by this comparison, one can define the \emph{tensor product} $\cM\otimes_\V \cN$ of $\V$-modules as the universal object that has such an intertwining operator
$$\Y_{\cM,\,\cN}:\; \cM\otimes \cN \to (\cM\otimes_\V \cN)\{z\}[\log(z)].$$
This was anticipated in the physics and analysis context in \cites{MS88,FF88, Verl88, Gab94}, appeared prominently for affine Lie algebras in \cite{KL93}, and was made rigorous in the vertex algebra settings in \cite{HL94} and \cite{HLZ06}. 
The tensor unit is the vertex algebra $\V$ itself, and the intertwining operator witnessing the isomorphisms $\V\otimes_\V\cM\cong \cM$ is precisely the action $\Y_\cM$, as for commutative algebras. We then encounter a major difference to the commutative algebra paradigm: An intertwining operator from $\cM,\cN$ does naturally induce an intertwining operator from  $\cN,\cM$, but the variable has to be replaced by $-z$. Since the involved functions are multivalued, this has to be done by analytic continuation,  counterclockwise around $z=0$, and doing this twice will not give the identity in general. Hence our category of representations comes with a non-symmetric braiding.

\begin{theorem}[\cite{HLZ06}]
For a $C_2$-cofinite vertex algebra $\V$, the category of $C_2$-cofinite modules is a braided tensor category. 
\end{theorem}

There is a natural notion of a dual, the \emph{contragradient} module $\cM^*$, but frequently it does not fulfill the axioms of rigidity. For example, the tensor unit may not be self-dual and the tensor product may not be exact, merely right-exact.

Note that in \cite{ALSW21}, which is joint work with R. Allen, C. Schweigert and S. Wood, we have shown that the contragradient dual fulfills the axioms of a \emph{Grothendieck-Verdier} category in the sense of \cite{BD13}. Technically, this type of duality is defined by being right adjoint to the tensor product, so it is more or less equivalent to having a right-exact tensor product, and it still provides a notion of inner Homs. However, the duality is not monoidal, so in a sense there are two tensor products with equal rights. Again, this notion already appears for commutative algebras, where tensoring over the algebra is typically only right-exact and the algebra as a module over itself is typically not self-dual. A first example (where exactness still holds) is the lattice vertex algebra, if the conformal structure is shifted by a background charge, and this will become relevant later.

\bigskip

Vertex algebras are motivated by physics, where they describe the holomorphic (chiral) part of a 2-dimensional quantum field theory with conformal symmetry. More precisely, to any Riemann surface $\Sigma$ with punctures decorated by vertex algebra modules there is a notion of \emph{conformal block} $\cZ(\Sigma)$, the solution of a system of differential equations, as in the introductory example \eqref{exm_KZequation}. This space carries an action of the group of diffeomorphisms of $\Sigma$, up to homotopies fixing the boundaries, which is the mapping class group. 

The space of intertwining operators corresponds to the space of conformal blocks associated to a sphere with three punctures, and the mapping class group is the braid group $\B_2$, which acts by the braiding of the category. The space of conformal blocks of the torus contains, roughly by glueing, the graded dimensions $\sum_{n\geq 0} \dim(\cM_n)t^n$ of vertex algebra modules. Here $t=e^{\i\pi\tau}$ is the dependency on the choice of complex structure on the torus,   which we may realize by $\C/(\Z+\tau\Z)$ for $\tau$ in the upper halfplane. The mapping class group of the torus is the modular group $\mathrm{SL}_2(\Z)$, which acts on the set of possible complex structures parametrized by $\tau$ as Möbius transformations. In turn, the graded dimensions of the modules transform as a vector-valued modular form, and the transformation behavior is visible on the category side. In the nonsemisimple case, this is more involved, because the set of traces does not span the conformal block, similar as the set of characters of a group (or Hopf algebra) does not span the space of class functions in nonsemisimple settings.  

We continue by first examples of vertex algebras:

\begin{example}
The \emph{Heisenberg vertex algebra} is, as graded vector space, a formal polynomial ring in countably many variables, which we  here denote in physics notation
$$\cH=\C[\partial\varphi,\partial^2\varphi,\ldots].$$
Similarly, we have the $n$-dimensional Heisenberg algebra $\cH^{\otimes n}$ with generators $\partial^k\varphi_\lambda$. The multiplication has the following distinct singular term, plus many other regular terms forced by the compatibility axioms
$$\Y:\;\partial\varphi_\lambda\otimes\partial\varphi_\mu\mapsto (\lambda,\mu)z^{-2}1+\cdots$$
for $(\lambda,\mu)$ the standard inner product of $\R^n$.
There are vertex algebra modules $\cH_\lambda=\cH^{\otimes n}e^{\lambda\varphi}$ for every $\lambda\in \R^n$, each generated by a formal symbols $e^{\lambda\varphi}$, on which the action has the following distinct singular term
$$\Y_{\cH_\mu}:\;\partial\varphi_\lambda\otimes e^{\mu\varphi}\mapsto (\lambda,\mu) z^{-1}e^{\mu\varphi}+\cdots$$
The space of intertwining operators between $\cH_\lambda,\cH_\mu$ and $\cH_\nu$ is zero unless $\lambda+\mu=\nu$, in which case it is $1$-dimensional and the distinct term is
$$\Y_{\cH_\lambda\;\cH_\mu \choose \cH_{\lambda+\mu}}:\;e^{\lambda\varphi}\otimes e^{\mu\varphi}\mapsto z^{(\lambda,\mu)}e^{(\lambda+\mu)\varphi}+\cdots$$
As a consequence, the tensor product is ${\cH_\lambda\otimes_\V \cH_\mu}=\cH_{\lambda+\mu}$ and braiding is $e^{\i\pi(\lambda,\mu)}$ due to the multivaluedness of $z^{(\lambda,\mu)}$. To summarize, we have an equivalence of braided tensor categories 
$$\Rep(\cH^{\otimes n})\cong \Vect_{\R^n}^{1,\sigma},\qquad \sigma(\lambda,\mu)=e^{\i\pi(\lambda,\mu)}$$
The graded dimension of any $\cH_\lambda$ is related to the Dedekind eta function. 

Physically, this vertex algebra is the chiral part of a free scalar bosonic field \text{$\varphi:\Sigma\to \R^n$}. Classically, such a field would be a solution to the wave equation $\Box\varphi(x,t)=0$. In quantum field theory, in the Feynman path integral approach, $\varphi$ can be regarded as a random field.  The probability (actually: amplitude) of a particular $\varphi$  is hereby determined by a Lagrangian action functional $\int_\Sigma |\nabla \varphi|^2$. The minimum is obtained for those $\varphi$ that solve the Euler-Langrange equation, which recovers the classical equation. 
The intertwining operators are literally the correlation functions of quantities depending on $\varphi$, for example differential polynomials, evaluated at three points $0,z,\infty$ of the sphere.
\end{example}
\begin{example}
For any choice of an even integral lattice $\Lambda\subset \R^n$, the Heisenberg vertex algebra can be extended by a commutative algebra to the \emph{lattice vertex algebra} $\V_\Lambda$, as discussed later in Example \ref{exm_latticeVOA}. The corresponding braided tensor category follows from this construction:
$$\Rep(\V_\Lambda)\cong \Vect_{\Lambda^*/\Lambda}^{\bar{\omega},\bar{\sigma}},$$
where $\bar{\omega},\bar{\sigma}$ are obtained from $(1,e^{\i\pi(\lambda,\mu)})$ by picking coset representatives, as discussed in Example \ref{exm_latticeVOA}.
The graded dimensions are Jacobi theta functions. Physically, this is a free scalar bosonic field compactified on an $n$-torus.
\end{example}
\begin{example}\label{exm_affineg}
Let $\hat{\g}_\kappa$ be an affine Lie algebra, constructed from a finite-dimensional semisimple complex Lie algebra $\g$ by taking $\g\otimes \C[t,t^{-1}]$ and then a central extension, parametrized by the level $\kappa$. If we denote $x_n=x\otimes t^n$ for $x\in\g,\,n\in\Z$, then we can turn suitable modules of $\hat{\g}$ into vertex algebra modules by setting 
$$\Y:\;x_{-1}.1\otimes v\mapsto \sum_{n\in\Z} (x_n.v) z^{-n-1}.$$
For example, for generic levels $\kappa$ we can obtain a vertex algebra structure on the irreducible Verma module $\V^\kappa(\g)$. Then the category of representations of $\V^\kappa(\g)$ is equivalent to representations of the quantum group $U_q(\g)$, this is known as \emph{Kazhdan-Lusztig correspondence} \cite{KL93}. In some sense, this is the conformal field theory picture behind the introductory example \eqref{exm_KZequation}. 
\end{example}
Note that this method of first defining a Lie algebra of \emph{mode operators $x_n$} and then in hindsight from this a vertex algebra is very effective. For example, the Heisenberg vertex algebra in the previous example can be constructed in a similar way from the Heisenberg Lie algebra, with generators $a_n,n\in\Z$ and relations $[a_n,a_m]=n\delta_{n+m=0}$.

\begin{example}[Orbifold]\label{exm_VOAGOrbifoldA}
For a finite group $G$ acting on a vertex algebra $\V$ we take $\cW=\V^G$ the vertex subalgebra fixed by $G$. Restricting $\V$-modules produces $\V^G$-modules, but not all of them. We can in additional consider over $\V$ twisted modules involving multivalued functions \cites{FLM88,DLM96}: A module $\cM\in\Rep_\cW(\V)$ is called \emph{$g$-twisted} if 
$$\Y_\cM(v,e^{2\pi\i}z)=\Y_\cM(gv,z).$$
In particular, the restriction to $\V^G$ is single-valued.
The categories of $g$-twisted modules supposedly form a $G$-crossed braided tensor category \cites{ENOM10, DN21} 
$$\bigoplus_{g\in G}  \cC_g,$$
extending the category of local modules $\cC_1=\Rep(\V)$. The decomposition behavior in the category of $\cW$-modules is given, rather easily, by $G$-equivariantization, following a Schur-Weyl type argument, and can be proven in good cases to produce all modules over $\V^G$.
\end{example}

Now I want to explain my program to construct vertex algebras, whose category of representations is given by quantum groups in the generalized sense explained in Section~\ref{sec_quantumgroups}. The main ingredients on the vertex algebra side are screening operators, which appear in abundance, maybe since the work of \cite{DF84}. Local screening operators facilitate actions of Lie algebras on vertex algebras, for example $\g$ acts on $\V^\kappa(\g)$ by screening operators. More complicated are non-local screening operators that involve multivalued functions:

\begin{definition}
Let $\V=\cH^{\otimes n}$ be the $n$-dimensional Heisenberg vertex algebra. For any $\alpha\in \R^n$, we define the \emph{screening operator} on any module $\cM$ 
$$\zem_\alpha:\;\cM\to \overline{\cH_\alpha\otimes_\V\cM},$$
by taking the intertwining operators $\Y_{\cH_\alpha,\cM}$, which is part of the definition of the tensor product $\cH_\alpha\otimes_\V\cM$, applying it to $e^\alpha\otimes m$ and integrating this as a multivalued function over  the unit circle, lifted to a path in the multivalued covering. Since this integral is nonzero not only for the power $z^{-1}$ as usual, but also for all noninteger powers $z^m$, the result is in the algebraic closure of $\cH_\alpha\otimes_\V\cM$ as a graded vector space.
\end{definition}
Note that, as defined, non-local screening operators acting on an arbitrary module have in general bad compatibility with the grading, the Virasoro action etc. Suitable powers are well behaved and appear in the Felder complex \cite{Fel89}, from the perspective of multivalued functions and cycles this is explained for example in \cite{TW13}. In view of the following results, these results correspond to the case $\sl_2$.

\bigskip

In \cite{Len21} my main result is proving that screening operators are described by Nichols algebras. This was an expectation of B. Feigin et. al. \cites{FF92,FGST05, FT10, AM14} in the case of quantum groups and of A. Semikhatov and I. Tipunin \cite{ST12} in the context of diagonal Nichols algebras, and they exposed me to these questions during my research stay in Moscow (2015-2016). 

\newcommand{\Fp}{\mathrm{F}}
\begin{theorem}\label{thm_screenings}
Let $\alpha_1,\ldots,\alpha_n\in\R^n$. Assume a technical condition that ensures convergence I call \emph{subpolarity} in \cite{Len21} Sec. 5.2; it holds for example if $(\alpha_i,\alpha_j)$ is a positive definite matrix with diagonal entries less then $1$. Then the screening operators can be composed (convergently) and fulfill the relations of the corresponding Nichols algebra 
$$\NicholsOf(X),\qquad X=\C_{\alpha_1}\oplus\cdots\oplus \C_{\alpha_n}.$$
\end{theorem}
\begin{proof}
We sketch the idea of proof: If we try to compute the composition of screening operators, we encounter in the simplest case
  $$\big(\zem_{e^{\alpha_1}}\cdots\zem_{e^{\alpha_k}}\big)e^\lambda
    =\Fp((\alpha_i,\alpha_j)_{ij},(\alpha_i,\lambda)_i)\cdot e^{\alpha_1+\cdots+\alpha_k+\lambda}
    +\text{ higher terms},$$
for certain generalizations of Selberg integrals \cite{Sel44}, see the survey \cite{FW08},   
$$\Fp((m_{ij}),(m_i)):=\int\cdots\int_{[0,{2\pi\i}]^n} \prod_i z_i^{m_i}\prod_{i<j}(z_i-z_j)^{m_{ij}}\; \mathrm{d} z_1\ldots \mathrm{d} z_n,$$
where  $m_i\in\C$ and $m_{ij}=m_{ji}\in\C$ for $1\leq i,j\leq n$ and the integral is over a certain lift of the $n$-torus $(S^1)^n$ to a cube $[0,2\pi\i]^n$ in the multivalued covering. Note that the integral has a-priori no reasonable relation to the same integral with some indices $i,j$ switched if $m_{ij}\not\in\Z$, because this changes the chosen contour.  

The assertion amounts to the proof that these Selberg integrals have linear relations according to the linear relations of monomials in the Nichols algebra $\NicholsOf(X)$. If we refer to the general definition of Nichols algebras in terms of quantum symmetrizers, then this can be concluded if $\Fp$ can be written as quantum symmetrizer of some other integral $\tilde{\Fp}$, and indeed we can find such an expression by subdividing $[0,2\pi\i]^n$ into $n!$ simplices $\Delta^\sigma,\sigma\in\S_n$ with fixed ordering of the angles of $z_1,\ldots,z_n$. 

See \cite{Len21} Example 5.21 for the quadratic relations explicitly worked out in this way, in this case the Selberg integral is simply the Euler Beta-integral.

I should mention that, graphically, these ideas were present in \cite{ST12} for Nichols algebras and independently in \cite{Ros97} for quantum groups, but in the current setup the vertex algebra part as well as the complex analysis part, including convergence, require additional ideas. On the other hand, note that there exist versions of the Selberg integral for other root systems \cites{TV03,War09} with relations to MacDonald polynomials. It would be a nice project to extend their results to generalized root systems of Nichols algebras. 
\end{proof}

At this point I can formulate the logarithmic Kazhdan Lusztig conjecture, in my personal perspective and generality:
\begin{conjecture}[Logarithmic Kazhdan Lusztig conjecture]\label{con_KL}
Let $\alpha_1,\ldots,\alpha_n\in\R^n$ and again assume the convergence condition in the previous theorem. Consider the subalgebra $\cW\subset \cH^{\otimes n}$ defined as intersection of the kernel of the screening operators $\zem_{\alpha_1},\ldots,\zem_{\alpha_n}$. Then conjecturally there is in good cases\footnote{Surely some finiteness will be required for this statement, moreover there are non-rigid examples where the functor cannot be expected to be faithful, see \cite{Len25} Section 8 for a discussion of counterexamples.}  an equivalence of  braided tensor categories
$$\Rep(\cW)\cong \cZ_\cC(\Rep(\NicholsOf(X))(\cC)),$$
where the right-hand side is the generalized quantum group in Section \ref{sec_quantumgroups} associated to the Nichols algebra of screenings in Theorem \ref{thm_screenings}, explicitly
$$\cC=\Vect_{\R^n}^{1,\sigma},\qquad X=\C_{\alpha_1}\oplus\cdots\oplus \C_{\alpha_n},\qquad
\sigma(\alpha_i,\alpha_j)=e^{\pi\i(\alpha_i,\alpha_j)},$$
\end{conjecture}

Using simple current extensions, these correspondences can easily be transported to other Cartan parts, for example a lattice vertex algebra $\V_\Lambda$ and a generalized quantum group over $\cC=\Vect_{\Lambda^*/\Lambda}^{\bar{\omega},\bar{\sigma}}$.
\medskip

\begin{example}\label{ex_KLQuantum}
For the data in Example \ref{exm_quantumBorelpart}, the conjecture says that there is an equivalence of braided tensor categories of representations of the \emph{Feigin-Tipunin algebra} $\cW_p(\g)$ and representations of the quasi-quantum group $\tilde{u}_q(\g)$. With the Cartan part in Conjecture~\ref{con_KL}, the conjecture relates the singlet vertex algebra to the unrolled quantum group. For $\g=\sl_2$ this is now a theorem \cites{FGST05, AM08, TW13, CGR20, CLR21, GN21, CLR23}.

Let us mention that there is an important alternative characterization of $\cW_p(\g)$ in terms of cohomology of bundles over the flag variety in \cite{FT10}, which has been proven meanwhile in \cites{Sug21,Sug23}.
\end{example}

In \cite{FL22} we have done more detailed work on examples for other Nichols algebras of diagonal type, and in Section \ref{sec_proof} we will discuss our recent proof of the Kazhdan-Lusztig conjecture in Example \ref{ex_KLQuantum} for $\sl_2$ in \cite{CLR23} and for general $\g$ in \cite{Len25}. 

Now we go beyond the situation covered by Conjecture \ref{con_KL} which were realizations in the lattice vertex algebra and Nichols algebras of diagonal type: I am  convinced that the assertion about the screening operators in Theorem \ref{thm_screenings} and the logarithmic Kazhdan-Lusztig conjecture should hold for arbitrary well-behaved vertex algebras $\V$. That is:

\begin{problem}\label{prob_KLnondiagonal}
 Whenever we pick a vertex algebra $\V$ and a choice of module $\cM=\cM_{1}\oplus \cdots\oplus \cM_n$, then the corresponding algebra of screening operators should be given by the Nichols algebra $\NicholsOf(\cM)$ of rank $n$ inside $\Rep(\V)$. Moreover, the kernel of screening vertex algebra $\cW\subset \V$ should in good cases have a braided tensor category of representations equivalent to a generalized quantum group 
$$\Rep(\cW)\cong \cZ_\cC(\Rep(\NicholsOf(X)(\cC)).$$
The first assertion, the generalization of Theorem \ref{thm_screenings}, seems to be within reach, but there are two conceptual problems:
\begin{itemize}
    \item Instead of working with explicit Selberg integrals, we have to use the machinery of differential equations with regular points in \cite{HLZ06}. In fact, imitating the proof above requires precisely all that we have generally: The tensor product given by a universal intertwiner and the braiding given by analytic continuation around $z=0$.
    \item In a general module, there is no preferred generator $a_i=e^{\alpha_i}$. In fact, we will have to consider integrals of intertwiners simultaneously for all $a_i\in \cM_i$, which are of course not independent. Then the notion of the kernel of screening operators has to be replaced by a notion of the kernel of intertwining operators.
\end{itemize}
\end{problem}

Even if a proof of such a far reaching conjecture is not in sight at the moment, or it is wrong altogether, it has the following interesting implications, which has already sparked different collaborations:
\begin{itemize}
\item It proposes for any extension of vertex algebras $\cW\subset \V$ given by screening, for example free field realizations, a corresponding modular tensor category of representations.
\item It proposes for any braided tensor category of the form $\cZ_\cC(\NicholsOf(X)(\cC))$ a realization by a vertex algebra $\cW$, assuming we have been given a realization of the Cartan part~$\cC$ by a vertex algebra $\V$.
\end{itemize}

\begin{problem}
One can ask if there exists any nonsemisimple modular tensor category which is \emph{not} of the form $\cZ_\cC(\Rep(\cB)(\cC))$ for some semisimple modular tensor category $\cC$. Equivalently stated: is there a semisimple modular tensor category in any Witt class? The quest to classify all modular tensor categories with a fixed ''semisimple part'' can be seen as the categorical and the braided version of the Andruskiewitsch-Schneider program discussed in Section \ref{sec_quantumgroups}.
\end{problem}

\section{Examples beyond quantum groups}

In \cite{FL22}, which was joint work with my PhD student I. Flandoli, contains first steps toward other finite-dimensional Nichols algebras of diagonal type. It also contains a quite extensive introductory Section 2 on Nichols algebras and Section 3 on screening operators, which may be helpful to the reader interested in more details than we explained at this point. 

In this article, we go through Heckenberger's list of finite-dimensional Nichols algebras of diagonal type and classify non-integral lattices $\Lambda$ that realize the braiding, in the sense that the Gram matrix of the lattice is $m_{ij}=(\alpha_i,\alpha_j)$ and the braiding is $q_{ij}=e^{\pi\i m_{ij}}$. We depict such a choice of realization by adding $m_{ij}$ to the $q$-diagram:
\begin{center}
\begin{tikzpicture}
		\draw (0,0)--(1.6,0);
		\draw (-0.1,0) circle[radius=0.1cm] node[anchor=south]{$ q_{11}$}  node[anchor=north]{$ m_{11}$}
		(1.7,0) circle[radius=0.1cm] node[anchor=south]{$ q_{22}$}node[anchor=north]{$ m_{22}$};
		\draw (0.8,0) node[anchor=south]{$ q_{12}q_{21}$}
        node[anchor=north]{$ 2m_{12}$};
\end{tikzpicture} 
\end{center}
We look for realizations $(m_{ij})$ that possess an additional invariance under reflection, see Definition 5.1, as formulated in \cite{ST12}. We find this condition also  very reasonable, although at present we have never applied it. It is interesting to note that this condition makes the solution unique, except for cases in which different series coincide. For example, at $q^4=1$ there is an isomorphism $u_q(\sl_3)\cong u_q(\sl_{2|1})$, but the realizations $(m_{ij})$ for these two series are different even for such this value of $q$ (where $m$ is half-integer valued):

\begin{center}
\begin{tikzpicture}
		\draw (0,0)--(1.6,0);
		\draw (-0.1,0) circle[radius=0.1cm] node[anchor=south]{$ q^2$}
        node[anchor=north]{$ 2m$}
		(1.7,0) circle[radius=0.1cm] node[anchor=south]{$ q^2$}
        node[anchor=north]{$ 2m$};
		\draw (0.8,0) node[anchor=south]{$ q^{-2}$}
        node[anchor=north]{$ -2m$};
\end{tikzpicture} 
\hspace{2cm}
\begin{tikzpicture}
		\draw (0,0)--(1.6,0);
		\draw (-0.1,0) circle[radius=0.1cm] node[anchor=south]{$ q^2$}
        node[anchor=north]{$ 2m$}
		(1.7,0) circle[radius=0.1cm] node[anchor=south]{$ -1$}
        node[anchor=north]{$ 1$};
		\draw (0.8,0) node[anchor=south]{$ q^{-2}$}
        node[anchor=north]{$ -2m$};
\end{tikzpicture} 
\end{center}
Another important point we discuss in this article is, that for such cases beyond $u_q(\g)$ the convergence condition in Theorem~\ref{thm_screenings} often does not hold in general any more. This means that for each relation of the Nichols algebra, the Selberg integrals have to be analytically continued: Their poles encode extensions of the Nichols algebras (called \emph{pre-Nichols algebra} in \cite{An16}), outside those poles Theorem \ref{thm_screenings} will still hold. We demonstrate this in \cite{FL22} Section 4 for the truncation relation and the quantum Serre relations of type $A_2$, but this work is by no means complete. 
%
Meanwhile, there has been significant progress on the analysis side \cite{Suss24}, so that the analytic continuation is now available to us. 

\begin{problem}
    Determine which Nichols algebra relations hold for which lattice parameters $(m_{ij})$ and compare this to lifting data of Nichols algebras. 
\end{problem}

I also list some cases for the Kazhdan Lusztig conjecture in the more general sense of Problem \ref{prob_KLnondiagonal}, which  are interesting and particularly tangible:

\begin{itemize}
    \item We start with a group orbifold of a lattice vertex algebra $\V=(\V_\Lambda)^G$, see Example~\ref{exm_VOAGOrbifoldA}. If we take a kernel by screening operators from the untwisted sector $\cC_1$, then this is equivalent to taking a kernel of these screening operators on $\V_\Lambda$ first and then orbifold the result by $G$, where $G$ acts on the Nichols algebra of those screenings. This case should be closely related to the Nichols algebra of quantum groups over nonabelian groups we have constructed in Section \ref{sec_quantumgroups}, and the proof of the logarithmic Kazhdan-Lusztig conjecture for $(\V_\Lambda)^G$ should follow from the one for $\V_\Lambda$. 
    
    Alternatively, we could take screening operators from the twisted sectors $\cC_g$ and I have no intuition what their Nichols algebras, kernel of 
    screenings etc. are. 
    
    \item We start with an affine Lie algebra $\hat{\g}_\kappa$ at positive integer level, so $\cC$ is the well-known finite semisimple modular tensor category, which has no fibre functor to $\Vect$. The smallest case is the category with simple objects $\unit,X$ and $X\otimes X=X+1$. In this setting, the screening operators have a-priori a distinct Lie-theoretic flavor, for example $\g$ acts on the top space of the vertex algebra and its modules.
    
    Nichols algebras in this category have not been studied, to my knowledge. They can be produced from a larger Lie algebra with a parabolic Lie subalgebra with Levi factor $\g$, and similarly for other Nichols algebras, as follows: For $J\subset I$ a subset of simple roots, write the coinvariants of $\NicholsOf(I)$ with respect to $\NicholsOf(J)$ as a Nichols algebra over $\NicholsOf(J)$ as in \cite{AHS10} and now take the semisimplification of the latter. Then the image of the coinvariants is a Hopf algebra (typically again a Nichols algebra) over the semisimplification of the quantum group associated to $J$. However, it is not clear if all finite-dimensional Nichols algebras arise this way.
\end{itemize}

\begin{problem}
Construct and classify finite-dimensional Nichols algebras of objects in the twisted sectors of a $G$-crossed extension of an abelian group, for example the Tambara-Yamagami category. Classify finite-dimensional Nichols algebras of objects in the semisimple modular tensor category $\Rep(\hat{\g}_\kappa)$, or equivalently the semisimplification of the quantum group $u_q(\g)$, beyond those obtained by semisimplifying Nichols algebras over $u_q(\g)$.
\end{problem}

\newcommand{\loc}{{loc}}
\newcommand{\spl}{split}

\section{Group symmetries and the big quantum group}

The Feigin-Tipunin algebras $\cW_p(\g)$ in Example \ref{ex_KLQuantum} conjecturally carry an action of  the Langlands dual Lie group $G^\vee$ by additional local screening operators. For simply-laced $\g$ this is rigorously established in \cite{Sug21}. For more general vertex algebras constructed as kernel of screenings I would expect a similar action of a certain Lie group associated to the so-called roots of Cartan type of the Nichols algebra.

\begin{question}
What are the group orbifolds of $\cW_p(\g)$ under the action of $G^\vee$? The result should be a $G^\vee$-crossed braided tensor category $\bigoplus_g \cC_g$, in some sense, where $\cC_g$ for $g\in G^\vee$ are the $g$-twisted modules, and then some $G^\vee$-equivariantization. As we shall see, this will lead to infinite-dimensional versions of the quantum group, and in this sense the logarithmic Kazhdan-Lusztig correspondence extends to these cases. 
\end{question}

For $g\in H^\vee$ in the Cartan subgroup, the orbifold is generically the singlet vertex algebra in Example \ref{ex_KLQuantum}, and the twisted modules are simply full Heisenberg vertex algebra modules with non-integral conformal weight. There are also results in \cite{ALM13} for nonabelian  finite subgroups of $\SL_2(\C)$ acting on the triplet algebra $\cW_p(\g)$.

In \cite{FL24}, which is joint work with B. Feigin, we approach this problem from different sides. The smallest case $\cW_2(\sl_2)$ corresponds to the vertex superalgebra of symplectic fermions, where we identify silently $\sl_2=\mathfrak{sp}_2$. Here we can write down explicitly a twisted mode algebra (in the sense of Example \ref{exm_affineg}), which has generators $\psi^\pm_k$ for $a\in\{\pm\}$ and $k\in\Z$, and relations 
$$\{\psi^a_m,\psi^b_n\}=m\begin{pmatrix} 0 & 1 \\ -1 & 0\end{pmatrix}_{a,b}\delta_{m+n=0}+\begin{pmatrix}
0& -1 \\ 1 & 0
\end{pmatrix}(Ae^a,e^b)\delta_{m+n=0},$$
where $A\in \sl_2$ such that $g=e^{2\pi\i A}$. From this we can determine easily all $g$-twisted module by inducing modules of the algebra generated by $\psi_0^+,\psi_0^-$, note that this algebra literally coincides with the small quantum group for $q^4=1$ and deformations thereof. New and  interesting are the cases where $A$ is nilpotent, which we treat in accordance with \cite{Bak16}. In this case the simple modules are not  direct sums of submodules of Heisenberg modules anymore, but extensions thereof, which piece together to Verma modules and coVerma modules, and ultimately to tilting modules, see \cite{FL24} Section 2.2 and the figures therein.   

We devise a method to construct $g$-twisted modules for general $\cW_p(\g)$: Any $g$ is contained in some Borel subgroup $B^\vee$, and we can choose a realization of $\cW_p(\g)$ as kernel of screening operators of a lattice vertex algebra $\V_\Lambda$, such that the action of $B^\vee$ on $\cW_p(\g)$ extends to the action of $B^\vee$ on $\V_\Lambda$ by local screening operators. We can now obtain $g$-twisted modules of $\cW_p(\g)$ by restricting $g$-twisted modules of $\V_\Lambda$ (we call this \emph{twisted free field realization}), moreover because the action on  $\V_\Lambda$ is inner, all such modules arise from Delta-deformation \cites{DLM96, Li97,FFHST02,AM09}.

The quantum group side of this construction should be as follows:

\begin{problem}
For the orbifold of $\cW_p(\g)$ under the action of $B^\vee$ we expect the $g$-twisted modules to correspond to modules over the quantum group with big center by Kac, DeConcini and Procesi 
$$\mathcal{O}((G^\vee)^*)\to U_q^{\mathrm{KdCP}}(\g)\to u_q(\g),$$
which has a large central subalgebra isomorphic to the ring of functions over the Poisson dual group, so its category of representations fibres over this group, and the fibre over $1$ is $u_q(\g)$. Here, we only take the fibres over  $B^\vee$ and we should refer to a version with associator, because the order of $q$ is even. In our article, we demonstrate this equivalence, which extends the logarithmic Kazhdan-Lusztig correspondence between $\cW_2(\sl_2)$ and $\tilde{u}_q(\sl_2)$, as abelian categories. 

The modules over the orbifold itself, which is the associated $\mathrm{W}$-algebra, should hence be related in this way to the  $G$-equivariantization, which are the modules over the mixed quantum group in the sense of D. Gaitsgory \cite{Gait21}. Note that the $\mathrm{W}$-algebra is usually attributed to the big quantum group of divided powers (which has a degenerate braiding and is ``smaller''), so this somehow suggest one should admit a more general type of modules, leading to the larger category.
\end{problem}

The actual punchline of \cite{FL24} is that we can construct correspondingly a vertex algebra with a big central subalgebra,  the algebra of $G^\vee$-connections on the formal punctured disc $\C^\times$. The construction uses a semiclassical limit $\kappa\to \infty$ of a generalized quantum Langlands kernel, somewhat similar to the affine Lie algebra at critical level having a big center. Up to gauge transformations $F(z)\in G^\vee((z))$, the connections with regular singularity at $z=0$ are classified by their monodromy $g\in G^\vee$ up to conjugation. It is also tempting to study connections with irregular singularities, as we have done recently in \cite{FL24b}, again in the case $\cW_2(\sl_2)$, and which produces Whittaker modules. There should be associated categories of quantum group representations, probably those beyond highest-weight modules. 

\section{Commutative algebras and a proof strategy}\label{sec_proof}

In this section I present my most recent results, which are ideas developed during my research stay at the University of Alberta and collaboration with T. Creutzig. It allowed us to conclude in \cite{CLR21} the proof of the logarithmic Kazhdan-Lusztig conjecture in the smallest case $u_q(\sl_2)$ at $p=2$, building on the work of many people that already set up the vertex tensor category, computed the abelian category etc. In \cite{CLR23} we have greatly improved our methods to the systematic setup below, which in particular implies the case $\sl_2$ at arbitrary $p$ and other small cases. Most recently, I could use this to prove the case $u_q(\g)$ with $\g$ simply laced, conditional on the category on the vertex algebra side being reasonable \cite{Len25}.

\bigskip

The main additional ingredient is to study extensions of braided tensor categories by commutative algebras, and compare them to extensions by Hopf algebras: Let $A$ be an algebra in a tensor category $\cD$, then there is a straightforward notion of a category of $A$-modules, which we call $\cD_A$, and a tensor category of $A$-$A$-bimodules ${_A}\cD_A$ with the tensor product $\otimes_A$. If $A$ is a commutative algebra in a braided tensor category $\cD$, then as in classical algebra every $A$-module can be turned into an $A$-$A$-bimodule, and $\cD_A$ becomes a tensor category with $\otimes_A$. However, in the braided case there are two inequivalent choices for such an identification, using over- or underbraiding. In turn, the tensor category $\cD_A$ itself is not braided, but there is a braided tensor subcategory $\cD_A^\loc$ consisting of \emph{local $A$-modules} $M$, which are modules such that the following diagram commutes:
\begin{equation*}
\begin{split}
\xymatrix{
A \otimes  M    \ar[rd]_{  \rho_M }  \ar[rr]^{ c_{M, A} \,\circ\,  c_{A, M}} && A\otimes M \ar[ld]^{\rho_M} \\
& M &  \\
} 
\end{split}
\end{equation*}
This appears early in \cites{Par95,KO02}, see \cites{FFRS06,SY24} for key properties of these categories.

We have an adjoint pair of functors: Left-adjoint the right-exact induction functor $A\otimes(-):\,\cD\to \cD_A$, which is a monoidal functor, and right-adjoint the left-exact functor forgetting the $A$-action $\mathrm{forget}_A:\,\cD_A\to \cD$, which is a lax monoidal functor. 
We summarize all structures in the following diagram:
$$
 \begin{tikzcd}[row sep=10ex, column sep=15ex]
   {\cD} \arrow[shift left=1]{r}{A\otimes(-)}  
   & 
   \arrow[shift left=1]{l}{\mathrm{forget}_A} 
   \cB  
   =\cD_A \\
   &\arrow[hookrightarrow, shift left=6]{u}{}
   \cC 
   =\cD_A^\loc
\end{tikzcd}
$$

The main example from our perspective is as follows \cites{HKL15, CKM24}: Let $\cW\subset \V$ be a conformal extension of vertex operator algebras. Then $A=\V$ can be regarded as a commutative algebra in the braided tensor category of modules over $\cW$. An $A$-module structure on a $\cW$-module $\cM$ corresponds, by definition of the tensor product, to an intertwining operator
$$\Y_\cM:\;\V\otimes \cM\to \cM\{z\}[\log(z)],$$
fulfilling certain axioms, such that the restriction to $\cW$ is single-valued. 
We call such a module \emph{twisted $\V$-module} over $\cW$ and denote their category by $\Rep_\cW(\V)$. The local $A$-modules are precisely those for which $\Y_\cM$ itself is single-valued (thus the name local), hence these are the true vertex algebra modules over $\V$. We thus have the picture

\begin{equation}\label{formula_VWdiagram}
 \begin{tikzcd}[row sep=10ex, column sep=15ex]
   \Rep(\cW) \arrow[shift left=1]{r}{A\otimes(-)}  
   & 
   \arrow[shift left=1]{l}{\mathrm{forget}_A} 
   \Rep_\cW(\V) \\
   &\arrow[hookrightarrow, shift left=0]{u}{}
   \Rep(\V)
\end{tikzcd}
\end{equation}

\begin{example}\label{exm_latticeVOA}
The lattice vertex algebra $\V_\Lambda$ is an extension of the Heisenberg algebra $\cH^{\otimes n}$ by a commutative algebra $A=\C_\epsilon[\Lambda]$. As an object, this is a direct sum of invertible objects $\C_\alpha,\alpha\in\Lambda$, and such cases are called \emph{simple current extension}. Now, for an even integral lattice there exists a $2$-cocycle $\epsilon$ that makes the twisted group ring $A$ commutative (in the braided sense). The twisted $\V_\Lambda$ modules are given by cosets $\lambda+\Lambda$, involving the multivalued function $z^{(\alpha, \lambda)}$ for $\alpha\in \Lambda$. To  define a braiding for cosets $\bar{\sigma}(\lambda+\Lambda, \mu+\Lambda)$ we have to choose representatives, and these choices in general prohibit $\bar{\sigma}$ from being bimultiplicative, and thereby cause an associator $\bar{\omega}$. The local modules are those where $(\lambda,\alpha)\in\Z$ for all $\alpha\in \Lambda$, meaning $\lambda$  is in the dual lattice $\Lambda^*$. Hence the picture is:
$$
 \begin{tikzcd}[row sep=10ex, column sep=15ex]
  \Vect_{\R^n}^{1,\sigma} \arrow[shift left=1]{r}{A\otimes(-)}  
   & 
   \arrow[shift left=1]{l}{\mathrm{forget}_A} 
   \Vect_{\R^n/\Lambda}^{\bar{\omega},\bar{\sigma}} \\
   &\arrow[hookrightarrow, shift left=0]{u}{}
   \Vect_{\Lambda^*/\Lambda}^{\bar{\omega},\bar{\sigma}}
\end{tikzcd}
$$
\end{example}
\begin{example}\label{exm_GorbifoldB}
Recall the situation in Example \ref{exm_VOAGOrbifoldA}, namely for a finite abelian group $G$ acting on a vertex algebra $\V$ we consider the orbifold vertex algebra $\cW=\V^G$. Then this is a particular type of extensions $\mathcal{W}\subset \V$ and the concepts we just discussed read as follows: The $g$-twisted modules in Example \ref{exm_VOAGOrbifoldA} are a particular type of twisted modules, with prescribed multivalued behavior, and the known theory concludes (essentially by a dimension argument) that the $g$-twisted modules for all $g\in G$ exhaust all twisted modules. The $1$-twisted modules are precisely the local modules.  The commutative algebra $A$ is the ring of functions $\mathcal{O}(G)$ as an object in the equivariantization:
$$
 \begin{tikzcd}[row sep=10ex, column sep=15ex]
   \bigoplus_{g\in G}  \cC_g /\!/ G \arrow[shift left=1]{r}{A\otimes(-)}  
   & 
   \arrow[shift left=1]{l}{\mathrm{forget}_A} 
   \bigoplus_{g\in G}  \cC_g\\
   &\arrow[hookrightarrow, shift left=0]{u}{}
   \cC_1
\end{tikzcd}
$$

\end{example}

In some sense, the logarithmic Kazhdan-Lusztig correspondence should be understood as a \emph{generalized orbifold with respect to the Nichols algebra of non-local screenings.} Proving the logarithmic Kazhdan-Lusztig conjecture amounts to matching the general picture of extensions \eqref{formula_VWdiagram} to the following corresponding picture for generalized quantum groups: In any $\cZ_\cC(\cB)$ there is a commutative algebra $A$  (the image of the tensor unit of $\cB$ under the induction functor in Section \ref{sec_quantumgroups}), such that the category of $A$-modules is  $\cB$ and the category of local $A$-modules is $\cC$. 

$$
 \begin{tikzcd}[row sep=10ex, column sep=15ex]
   \cZ_\cC(\cB) \arrow[shift left=1]{r}{\mathrm{forget}_{c_{X,-}}}  
   & 
   \arrow[shift left=1]{l}{} 
   \cB\\
   &\arrow[hookrightarrow, shift left=0]{u}{}
   \cC
\end{tikzcd}
$$

\begin{example}
For the small quantum group we have the following picture. Note that induction and restriction are reversed compared to commutative algebras, and that the identification of $\Rep(A),\otimes_A$ with representations over a Hopf subalgebra $u_q(\g)$ is not obvious at all, but holds in a general Hopf algebras setting \cites{Tak79, Skry07}

$$
 \begin{tikzcd}[row sep=10ex, column sep=15ex]
   \Rep(u_q(\g))
   \arrow[shift left=1]{r}{\mathrm{res}}  
   & 
   \arrow[shift left=1]{l}{\mathrm{ind}} 
   \Rep(u_q(\g)^{\geq 0})\\
   &\arrow[hookrightarrow, shift left=0]{u}{}
   \Rep(u_q(\g)^{0})
\end{tikzcd}
$$

\end{example}

In the preprint \cite{CLR23}, which is joint work with T. Creuztig and M. Ruppert, we set up the problem in this manner: The goal is to show $\Rep(\cW)\cong\cZ_\cC(\cB)$ from the known category $\Rep(\V)\cong \cC$ and additional information. First, we solve in general the problem of reconstructing the horizontal arrow for an arbitrary commutative algebra, by further developing in \cite{CLR23} Section 3 a functor devised by P. Schauenburg \cite{Sch01}.

\begin{theorem}\label{thm_Schauenburg}
Let $\cD$ be a finite rigid braided tensor category and  $A$ be a haploid commutative algebra. Then the induction functor $A\otimes (-)$ to $\cB=\cD_A$ upgrades to a fully faithful functor to $\cZ(\cD_A)$, whose image is precisely $\cZ_\cC(\cD_A)$ with $\cC=\cD^\loc$. 
\end{theorem}
Hence, $\cD=\Rep(\cW)$ can be reconstructed from the category of twisted modules $\cB=\Rep_\cW(\V)$, as in the group orbifold case it can be reconstructed by $G$-equivariantization. 

We then use in this article the existing knowledge about the braided tensor category of representations $\cD$ of the triplet vertex algebra  ($C_2$-cofiniteness, abelian category, tensor products of simple objects) to compute in Section~7 the abelian structure of the category of twisted modules $\cD_A$. 

We find that the only simple modules in $\cD_A$ are the local modules in $\cD_A^\loc$, which is very different from the $G$-orbifold case. Differently spoken, any object in $\cD_A$ has a composition series of local modules. From this we  find a splitting tensor functor $\cD_A\to \cD_A^\loc$ that sends every module to the associated graded module, which is local. By our Theorem \ref{thm_LM} this implies\footnote{In the discussed article this result was not yet available to us, we argued more ad-hoc using the known abelian structure.} that there exists a Hopf algebra $\Nichols$ such that
$$\cB \cong \Rep(\Nichols)(\cC).$$ 
By additional arguments (using that we have previously established that the triplet algebra category can be realized by a quasi-Hopf algebra and by grading arguments) we find that $\Nichols$ must be the smallest Nichols algebra discussed in Example \ref{exm_NicholsRank1},
$$\Nichols=\C[x]/x^p=\tilde{u}_q(\sl_2)^+.$$
From there we can reconstruct $\cD$ from Theorem \ref{thm_Schauenburg} and with the categorical description of the quantum group $\tilde{u}_q(\sl_2)$ in Section \ref{sec_quantumgroups} we have:

\begin{theorem}
The logarithmic Kazhdan Lusztig conjecture holds in the case $\sl_2$, and some similar cases in rank $1$, for example $\mathfrak{gl}(1|1)$.
\end{theorem}

In my recent preprint \cite{Len25} I want to push this approach to the general case described in Conjecture~\ref{con_KL}, without prior knowledge of the abelian category. The main problem is that on the vertex algebra side there is not enough information about the kernel of screening to even prove the existence of the vertex tensor product, let alone rigidity, although the methods have recently improved drastically, both on the vertex algebra side, see e.g. \cites{CMY23,NORW24}, 
and categorical side \cite{CMSY24}. So the following algebraic ideas are conditional on the structure on the vertex algebra side being suitably well-behaved:
\begin{itemize}
\item We have meanwhile proven in \cite{LM24} the general reconstruction result in Theorem \ref{thm_LM}, which guarantees a realizing Hopf algebra in $\cC$ in split situations.
\item Even if we do not know whether $\cC\hookrightarrow\cB$ admits a splitting, we can consider the subcategory $\cB^{\spl}\subset \cB$ consisting of objects with composition series in $\cC$. Then this can be realized by some Hopf algebra $\Nichols$ (we switch to a dual setup of corepresentations for technical reasons)
$$\cB^{\spl}\cong \CoRep(\Nichols)(\cC)$$
Moreover, $\Nichols$ is connected as a coalgebra (this is the dual notion to a basic algebra).
\item Already knowing $\Ext_\cB^1(1, \C_a)$ for simple objects in $\cC$ gives a lot of information: It gives a Hochschild $1$-cocycle, and in the dual setup this is a primitive element $x\in \Nichols$. This means by the universal property of the Nichols algebra that $\Nichols$ has to contain some deformed (i.e. not necessarily graded) version of the Nichols algebra
$$\NicholsOf(X)
\quad \text{for}\quad 
X=\bigoplus_{a\in \Gamma}\Ext_\cB^1(\unit, \C_a)\C_a$$
\item At this point I use deep results on the Hopf algebra classification side: As shown by I. Angiono et. al. \cite{An13} for finite-dimensional Nichols algebras of diagonal type, using their classification, there is no deformation and there holds generation in degree $1$, that is 
$$\Nichols\cong\NicholsOf(X).$$
\item Finally, I show  that for $\cW$ a kernel of screening operators $\zem_{\alpha_1},\ldots,\zem_{\alpha_n}$ there are in $\Rep_\cW(\V)$ extensions of the simple objects $1\to \cM\to \C_{\alpha_i}$ for each $\alpha_i$, essentially by using the information from the $\sl_2$-case 
\end{itemize}

\begin{lemma}
For $\cW$ the kernel of screening operators as in Conjecture \ref{con_KL}, if we assume that all involved tensor categories exist and are finite and rigid, then  $\Rep_\cW(\V)$ has a subcategory $\Rep(\NicholsOf(X))$ according to the conjecture and $\Rep(\cW)$ has a subquotient according to the conjecture. 
\end{lemma}

To show that this functor is actually an equivalence of categories, there are different options we describe. The easiest one is to use the graded dimensions, which is known for the cases related to simply-laced Lie algebras, because here the Feigin-Tipunin vertex algebra has a realization as sections of a vertex algebra bundle over the flag variety, and the Lefschetz fixpoint formula gives a formula for the graded dimension \cites{FT10,Sug21}. It is expected for $C_2$-cofinite vertex algebra (and proven for rational vertex algebras) that the asymptotics of the graded dimension of a module, as an analytic function,  is related to the Frobenius-Perron dimension of the module. The computation in \cite{BM17} do agree with the dimension of the quantum group module. Thus we can conclude in such a case:

\begin{theorem}[\cite{Len25}]
For $\cW$ the kernel of screening operators as in Conjecture \ref{con_KL}, for $\g$ simply-laced,  if we assume that all involved tensor categories exist and are finite and rigid, and the asymptotic of the graded characters agrees with the Frobenius-Perron dimension, then the conjecture holds. 
\end{theorem}

\begin{problem}
Actually, there is a direct way to see how the algebra of screening operators appears in the category of twisted modules, namely through the process of Delta-deformation \cites{DLM96, Li97}, now by non-local screening operators. This method was used in \cites{FFHST02,AM09} to construct modules of $\cW$, but it seems quite clear that the output of the construction is actually a twisted $\V$-module. This leads to a  functor
$$\Rep(\NicholsOf(X))(\cC)\to \cB,$$
which should be fully faithful, but again essential surjectivity is not clear. 
\end{problem}

Let us also briefly mention other types of algebraic considerations, that are relevant to the picture:
\begin{itemize}
\item The associated topological field theory (more precisely:  modular functor) has in the nonsemisimple case a derived version. We discuss this in the research monograph \cite{LMSS23}. There are apparent connections to the logarithmic conformal field theory picture, but those are not clear yet. 
\item A conformal field theory comes with boundary conditions and defects, corresponding to module categories and bimodule categories, and in particular the Brauer-Picard group. In \cite{LP17} we have constructed in a general setup three familes of bimodule categories which we think might exhaust all possibilities, but I would still consider these questions to be widely open.
\end{itemize}

\end{document}